\documentclass[11pt]{article}
\usepackage{fullpage}

\pdfoutput=1
\usepackage{graphicx, rotating}
\usepackage{latexsym,amsmath,amssymb,amsthm,epic,eepic,multirow}
\usepackage{natbib}
\usepackage{verbatim} 
\usepackage{url}
\usepackage{setspace}
\usepackage{graphicx}
\usepackage{multirow}
\usepackage{subfigure}
\usepackage{amsfonts}
\usepackage{natbib}
\usepackage{algorithm}
\usepackage{algorithmic}
\usepackage[pdftex]{color}
\usepackage{colortbl}
\definecolor{first}{rgb}{0.65,0.65,0.65}
\definecolor{second}{rgb}{0.9,0.9,0.9}

\newcommand{\PP}{\mathbb{P}}
\newcommand{\EE}{\mathbb{E}}

\newcommand{\eps}{\varepsilon}

\newcommand{\Cov}{\mbox{Cov}}

\newcommand{\bx}{\mathbf{X}}
\newcommand{\by}{\mathbf{Y}}

\newcommand{\rank}[1]{\mbox{rank}}
\newcommand{\RR}{\mathbb{R}}
\newcommand{\si}{\mbox{si}}
\newcommand{\ch}{\mbox{ch}}
\newcommand{\pa}{\mbox{pa}}
\newcommand{\tr}{\mbox{tr}}

\newtheorem{theo}{Theorem}

\newtheorem{lemm}{Lemma}
\newtheorem{corr}{Corollary}

\usepackage{xr}

\begin{document}
\title{Hierarchical Testing in the High-Dimensional Setting with
  Correlated Variables} 

\author{Jacopo Mandozzi and Peter B\"uhlmann\\
Seminar for Statistics, ETH Z\"urich}

\maketitle

\begin{abstract}
 
We propose a method for testing whether hierarchically ordered groups of
potentially correlated variables are significant for explaining a response
in a high-dimensional linear model. In presence of highly
correlated variables, as is very common in high-dimensional data, it seems
indispensable to go beyond an approach of inferring individual regression
coefficients, and we show that detecting smallest groups of variables (MTDs:
minimal true detections) is realistic. 
Thanks to the hierarchy among the groups of variables, powerful
multiple testing adjustment is possible which leads to a data-driven choice
of the resolution level for the groups. Our procedure, based on repeated
sample splitting, is shown to asymptotically control the familywise error
rate and we provide empirical results for simulated and real data which
complement the theoretical analysis. Supplementary materials for this
  article are available after the References.
\end{abstract}


\section{Introduction}

High-dimensional statistical inference where the number $p$ of
(co-)variables might be much larger than the sample size $n$ has become
a key issue in many areas of applications. We focus here on the linear model 
\begin{eqnarray}\label{mod1}
\by = \bx \beta^0 + \eps,\ \eps \sim {\cal N}_n(0,\sigma^2 I)
\end{eqnarray}
with $n \times p$ design matrix $\bx$, $p \times 1$ regression vector
$\beta^0$ and $n \times 1$ response $\by$, allowing for high-dimensionality
with $p \gg n$. Often, the 
active set of variables carrying the relevant information $$S_0 = \{ j;
\beta^0_j \neq 0 \}$$  
is assumed to be a small subset of all variables, i.e., the model is sparse
with many $\beta^0_j$ being equal to zero. Our main goal is testing of
significance of groups of parameters: 
for a group or cluster $C \subseteq \{1,\ldots ,p\}$,
\begin{eqnarray*}
H_{0,C}:\ \beta^0_j = 0\ \mbox{for all}\ j \in C,\ \ \ H_{A,C}:\ \beta^0_j
\neq 0\ \mbox{for at least one}\ j \in C. 
\end{eqnarray*}

Significance testing in the high-dimensional framework is essential when
looking beyond point estimation. \citet{WR08} propose an approach based on
single sample splitting, and \citet{memepb09} improve the reliability and
power of the method based on multiple sample splitting. \citet{mititi11} consider a perturbation technique, and (modified)
bootstrap-type schemes are analyzed by \citet{chala13} and \citet{liuyu13}. 
Another line of
methods have been proposed using low-dimensional regularized projections
(e.g. on single variables for individual hypotheses $H_{0,j}$) which have
some optimality properties
\citep{zhazha13,bue12,jamo13,geeretal13,javmo13}. However, in presence of highly correlated variables, all
these methods are likely to fail for testing individual hypotheses $H_{0,j}$.  

An interesting way to address the fundamental limitation of
identifiability in presence of high correlation or near linear dependence
is given by a hierarchical testing scheme proposed by \citet{Meins08}. 
First, the variables are grouped in a hierarchical way, for example by
hierarchical clustering. At the top of the hierarchy, the
global hypothesis $H_{0,\{1,\ldots ,p\}}$ is tested. If it can be rejected, 
a finer partition with clusters $\{C_k\}_k$ is considered, and for the ones
where $H_{0,C_k}$ can be rejected, one proceeds down the hierarchy
to finer partitions. The method has the powerful advantage that it
  \emph{automatically} goes (from top to bottom in the hierarchy) to finer
  resolution with smaller clusters, depending on signal-strength and the
  correlation structure among the variables. At the end, significant
  clusters can be typically found, and if the signal for an individual
  variable is sufficiently strong, even significance of a single variable
  can be detected. 
  \citet{Meins08} has worked out a simple yet powerful way 
  for controlling the familywise error rate when performing multiple tests
  in the hierarchy, assuming that there is a method which leads to valid
  p-values for the various hypotheses tests; for example, when $p < n$ and
  with Gaussian errors, one can use partial F-tests. 

\subsection{Our contribution}

As one of our main contributions, we deal here with the problem to obtain valid
p-values for hypotheses $H_{0,C}$ where $C$ is an arbitrary group of
(typically highly correlated) variables, for the high-dimensional scenario
where $p \gg n$. We address this important and open issue; note that
testing the global null-hypothesis $H_{0,\{1,\ldots ,p\}}$, in contrast to
the ``partial'' hypothesis $H_{0,C}$ for some $C$ with cardinality $1< |C| <
p$, is a rather different issue and has been addressed before
\citep[cf.]{goeman2006testing}.
Once we have valid p-values for $H_{0,C}$ for a arbitrary groups $C$, we
make use of the method from \citet{Meins08} leading to 
non-asymptotic bounds for strong 
control of the familywise error rate in a hierarchical structure.
For construction of the p-values, we rely on multiple sample splitting 
\citep{memepb09}. While this might be sub-optimal from a theoretical 
perspective, especially with respect to power, the method seems to perform 
well in a larger empirical study for individual hypotheses 
$H_{0,j}\ (j=1,\ldots ,p)$ in terms of reliable control of the familywise error
rate in multiple testing \citep{DeBuMeMe14}. 
We also extend the Shaffer improvement in \citet[Section 3.6]{Meins08} 
to the high-dimensional scenario, increasing the power of the hierarchical
method such that detection of more singletons than with the method from
\citep{memepb09} becomes possible.

Our second main contribution is the development of new methodology and
theory for hierarchical inference and testing of hypotheses, using
multiple sample splitting techniques (and \emph{multiple} sample splitting is
important for reproducibility \citep{memepb09}). Regarding 
methodology, the hierarchical approach allows for a substantially higher
number of so-called minimal true detections (MTDs: significant smallest
groups of variables) than the single variable analogue and has the
remarkable property of adaptively selecting a best resolution level
(MTDs with the smallest possible cardinality). 
We prove strong control of the familywise error rate of the hierarchical
method under a 
``zonal assumption'' which is weaker than the standard $\beta$-min condition
(used in \citet{memepb09}), that is, we do not require that all non-zero 
coefficients in the parameter-vector $\beta^0_j$ are sufficiently large. We 
demonstrate the finite sample behavior of the method with various empirical
results.  

We note that recently, \citet{meins13} describes another procedure for
dealing with highly correlated variables and hierarchical testing of groups
or clusters of variables. His method is an interesting alternative with the
remarkable property that it does not require (major) regularity assumptions
on the design matrix. The procedure is taking advantage of the special
structure of a linear model while our approach is: (i) more generic and
conceptually applicable to other (e.g. generalized linear) models, and (ii)
computationally much more efficient due to variable screening in a first
stage.   

\subsection{Outline of the paper}

In Section \ref{sectiondescription} we describe our method for obtaining
p-values for groups of variables and its use for hierarchical testing in
high-dimensional settings. We show in Section 
\ref{sectionFWER} that the familywise error rate (FWER) is strongly controlled,
and we describe a Shaffer improvement to increase the method's power while
keeping control over the FWER. Section 
\ref{sectionempirical} is devoted to empirical results: we show that
our procedure improves the single variable testing method of \citet{memepb09} in
settings with strong correlation among certain variables, particularly with
respect to minimal true detections (MTDs). In Section \ref{sectionrobust}
we provide theoretical evidence that the FWER is controlled even 
if a ``screening assumption'' required in Section \ref{sectionFWER} is not
satisfied. 

\section{Description of method}\label{sectiondescription}


Our method is based on four main steps: (i) hierarchical clustering of the
variables, (ii) variable screening in a linear model, (iii) significance
testing (with multiplicity adjustment) based on sample splitting, and (iv)
aggregation over multiple sample splits and hierarchical multiplicity
adjustment. See also Section \ref{subsec.schemesummary} for a schematic
summary.

\subsection{Clustering}\label{subsec.clustering}
In a first step, we construct a hierarchy of clusters. A
hierarchy, which can be represented as a tree-graph, $\mathcal{T}$ is a set
of clusters $\{C_k\}_k$ with $C_k 
\subseteq \{1,\dots,p\}$: the root node of the tree $\{1,\dots,p\}$ contains all
variables and for any  two clusters $C_k,C_{k'} \in \mathcal{T}$, either
one cluster is a subset of the other, or they have an empty
intersection. We use the notation $\pa(C)$ for the parent of a cluster $C$,
(the smallest superset of $C$), $\ch(C)$ for the children of a
cluster $C$ (all clusters that have $C$ as parent). Cluster $C$ is called an
ancestor of cluster $D$ if $D \subset C$.

As noted in \citet{Meins08}, the hierarchy can be derived from specific
domain knowledge or in some other natural way. The
philosophy of the method is that highly correlated variables (or variables
which are nearly linearly dependent) should end up in a
single small cluster: it will then be relatively easy to identify the cluster as
relevant, if it contains at least some variables from the active set 
$S_0$.
For our empirical results, we consider standard hierarchical clustering
based on correlation between variables, or a novel
hierarchical scheme using canonical correlations between clusters
\citep{buru12}. 

Once the hierarchical structure is given, the method goes on with a
hierarchical version of the multi sample-splitting procedure from
\citet{memepb09}. The following two steps, described in Sections
\ref{sec:Screening} and \ref{sec:TestandMulti} have to be repeated
for each sample split, indexed by $b=1,\ldots, B$ where $B$ is the number
of sample splits (since $B > 1$, we use the terminology multi
sample-splitting). 

\subsection{Screening}\label{sec:Screening}
The original data of sample size $n$ is split into two disjoint groups,
$N_{in}^{(b)}$ and 
$N_{out}^{(b)}$, i.e. a split $\{1,\dots,n \} = N_{in}^{(b)} \cup
N_{out}^{(b)}$ is chosen. The groups are chosen of equal size if $n$ is even
or satisfy $|N_{out}^{(b)}|=|N_{in}^{(b)}|+1$ if $n$ is odd.

Then, using only $N_{in}^{(b)}$, estimate with a screening procedure the
set of active predictors $\hat{S}^{(b)}$. A prime example is the Lasso
\citep{T96}.

\subsection{Testing and multiplicity adjustment}\label{sec:TestandMulti}

By considering for each cluster $C$
in the hierarchy $\mathcal{T}$ its intersection with $\hat{S}^{(b)}$, an
induced hierarchy with root $\hat{S}^{(b)}$ is given. Due to this
construction, assuming that the cardinality $|\hat{S}^{(b)}| < n/2$, the
situation is not high-dimensional anymore. Therefore, on this induced
hierarchy, we can apply testing procedure
similar as in \citet{Meins08}, the 
difference being that the hierarchical adjustment is not performed at this
stage but in Section \ref{sec:aggregating} after the aggregation over many
sample splits.

Based on the other half of the sample $N_{out}^{(b)}$, we use the classical
partial F-test with the full model $\hat{S}^{(b)}$ and submodel $C \cap
\hat{S}^{(b)}$ for the null hypothesis $H_{0,C \cap \hat{S}^{(b)}}$, where
$C \in {\cal T}$ is a given cluster. Thereby, we implicitly
assume that the submatrix $\bx_{\hat{S}^{(b)}}$ with columns corresponding
to $\hat{S}^{(b)}$ is of full rank (since $|\hat{S}^{(b)}| < n/2$). We then
assign the p-value from the partial F-test to the entire cluster $C$,
although we have only used the variables in $C \cap \hat{S}^{(b)}$. If a
cluster $C$ does not 
contain selected variables from $\hat{S}^{(b)}$, we set the p-value to
1. In summary, we define:
\begin{equation}\label{eq:pv}
p^{C,(b)} = 
\begin{cases} p_{\mathrm{partial \: F-test}}^{C \cap S^{(b)}}\
  \mbox{based on}\ \by_{N_{out}^{(b)}},\bx_{N_{out}^{(b)},\hat{S}^{(b)}} &,
  \text{if $C \cap \hat{S}^{(b)} \neq \emptyset$},\\
1 &, \text{if $C \cap \hat{S}^{(b)} = \emptyset$,}
\end{cases} 
\end{equation}
Then, for $C \in \mathcal{T}$ define the multiplicity adjusted (non-aggregated)
p-value as
\begin{equation}\label{multiplicityadj}
p_{adj}^{C,(b)} = \min \big(\, p^{C,(b)}
\frac{|\hat{S}^{(b)}|}{|C \cap \hat{S}^{(b)}|}~,~1 \big)
\end{equation}
if $C \cap \hat{S}^{(b)} \neq \emptyset$ and $p_{adj}^{C,(b)} = 1$ otherwise.

\subsection{Aggregation and hierarchical adjustment}\label{sec:aggregating}
By repeating the steps in Section \ref{sec:Screening} and
\ref{sec:TestandMulti} for $b = 1,\dots,B$, we obtain for each
cluster $C$ of the hierarchy ${\cal T}$ a set of $B$ p-values 
$p_{adj}^{C,(1)},\ldots ,p_{adj}^{C,(B)}$. We aggregate these p-values by
considering their empirical quantile. 

For $\gamma \in (0,1)$ define the aggregated p-values
$$Q^C(\gamma) = \min \big\{ \,1~,~q_\gamma \big( \big\{p_{adj}^{C, (b)} / \gamma;\, b=1,\dots,B \big\} \big) \big\},$$
where $q_\gamma(\cdot)$ is the (empirical) $\gamma$--quantile
function. Finally, define the hierarchically adjusted (aggregated) p-values as
$$Q_h^C(\gamma) = \max_{D \in \mathcal{T}:C \subseteq
  D} Q^D(\gamma)$$ 
such that the hierarchically adjusted (aggregated) p-value of a cluster
$C$ is always bigger than the hierarchically adjusted (aggregated) p-value
of an ancestor cluster.
In Section \ref{sectionFWER} we show that for any fixed $\gamma \in (0,1)$
the $Q_h^C(\gamma)$ are correct p-values. At this stage, $\gamma$ should be
considered as a pre-specified parameter of the method.  

Similarly as in \citet{memepb09}, error control is not guaranteed if we
optimize over $\gamma$, that is, for
each $C$ we would choose the minimal $Q_h^C(\gamma)$. Nevertheless, it is
possible to eliminate parameter $\gamma$ by proceeding as follows. Define 
\begin{equation}\label{defiPC}
P^C = \min \big\{ \,1~,~(1-\log\gamma_{\min}) \inf_{\gamma \in
  (\gamma_{\min},1)} Q^C(\gamma) \big\},
\end{equation}
for a lower bound $\gamma_{\min} \in (0,1)$ for $\gamma$, typically
$\gamma_{\min} = 0.05$. Then proceed with the hierarchical adjustment of
$P^C$ by defining
$$P_h^C = \max_{D \in \mathcal{T}:C \subseteq
  D} P^C.$$ 
These values $P_h^C$ are the final output of our method: we will show
again in Section \ref{sectionFWER} that $P_h^C$ are a valid 
p-value controlling the familywise error rate when testing over all $C \in
\mathcal{T}$. 

Our proposed ``top-down'': method is schematically summarized in Section
\ref{subsec.schemesummary} below. In the Supplemental Material we
illustrate a sub-ideal alternative ``bottom-up'' approach which is
empirically found to exhibit substantially less power. 

\subsection{Schematic summary of the method}\label{subsec.schemesummary}

We summarize our proposed method with the following schematic
description.

\begin{description}
\item[Step 1: Clustering]
$$\{\bx_1,\dots,\bx_p\}
\hspace{4mm} \xrightarrow{\mbox{clustering}} \hspace{4mm}
\mathcal{T}$$
Repeat for $b=1,\dots,B$:
\item[Step 2: Screening]
$$\{1,\dots,n\}=N \hspace{4mm}\xrightarrow{\mbox{sample split}}\hspace{4mm}  
N_{in}^{(b)} \cup N_{out}^{(b)} 
\hspace{4mm}\xrightarrow{\mbox{screening}}\hspace{4mm} 
\hat{S}^{(b)}$$
\item[Step 3: Testing and multiplicity adjustment]
$$|\hat{S}^{(b)}|<|N_{out}^{(b)}| 
\hspace{4mm}\xrightarrow{\mbox{testing}}\hspace{4mm}
p^{C,(b)}
\hspace{4mm}\xrightarrow{\mbox{multiplicity adjustment}}\hspace{4mm}  
p_{adj}^{C,(b)}$$
End of repeating for $b=1,\dots,B$.
\item[Step 4: Aggregation and hierarchical adjustment]
$$p_{adj}^{C,(b)}
\hspace{4mm}\xrightarrow{\mbox{aggregation}}\hspace{4mm}
Q^C(\gamma)
\hspace{4mm}\xrightarrow{\mbox{hierarchical adjustment}}\hspace{4mm}
Q^C_h(\gamma)$$ 
$$p_{adj}^{C,(b)}
\hspace{4mm}\xrightarrow{\mbox{aggregation}}\hspace{4mm}
Q^C(\gamma)
\hspace{4mm}\xrightarrow{\mbox{elimination of }\gamma}\hspace{4mm}
P^C 
\hspace{4mm}\xrightarrow{\mbox{hierarchical adjustment}}\hspace{2mm}
P^C_h$$ 
\end{description}

%




\section{Familywise error rate control}\label{sectionFWER}
We show in this section, that if the variable selection procedure
$\hat{S}$ satisfies two assumptions, then the p-values $Q_h^C(\gamma)$ and
$P_h^C$ defined in Section \ref{sectiondescription} control the familywise
error rate.
The assumptions are: 
\begin{eqnarray*}
&&\mbox{(A1) \it{Sparsity property}:}\; |\hat{S}| < n/2.\\
&&\mbox{(A2) \it{$\delta$-Screening property}:}\; \PP[\hat{S} \supseteq
S_0] \geq 1-\delta,\ \mbox{where}\ 0 < \delta < 1.
\end{eqnarray*}
The \textit{sparsity property} in (A1) implies that for each sample split
$b$ it holds that $|\hat{S}^{(b)}|<|N_{out}^{(b)}|$, a condition which is
necessary to apply classical tests. The \textit{$\delta$-screening
  property} in (A2) ensures that all the 
relevant variables are retained with high probability ($\delta$ is
typically small). While (A1) is the same condition as in
\citet[Section 3.1]{memepb09}, we consider with (A2) a slight modification
of the assumption in \citet[Section 3.1]{memepb09} in order to obtain
non-asymptotic bounds for familywise error rate control. We provide a
relaxation of the screening property (A2) in Section \ref{sectionrobust}.  

\medskip\noindent
\textbf{Example.} Consider the Lasso as a variable selection method
$\hat{S}$. Assumption (A1) holds for any value of the regularization
parameter. Assumption (A2) is ensured when requiring the following
conditions:
\begin{enumerate}
\item The design matrix $\bx$ satisfies the compatibility condition with
  compatibility constant $\phi_0^2$ \citep[cf. (6.4)]{pbvdg11}. Furthermore,
  it is normalized such that each column $\bx^{(j)}$ satisfies
  $\|\bx^{(j)}\|_2^2/n = 1$ for all $j=1,\ldots ,p$. 
\item A beta-min condition holds (we use here the notation $s_0=|S_0|$):
\begin{eqnarray*}
\min_{j \in S_0}|\beta^0_j| > 16 \sigma \sqrt{\frac{t^2 + 2 \log(p)}{n}}
s_0/\phi_0^2. 
\end{eqnarray*}
\end{enumerate}
Then, the Lasso with regularization parameter $\lambda = 4 \sigma_{\eps}
\sqrt{\frac{t^2 + 2 \log(p)}{n}}$ satisfies (A2) with $\delta = 2
\exp(-t^2/2)$ \citep[Lem. 6.2, Thm. 6.1, (2.13)]{pbvdg11}. 

\bigskip
Especially when the correlation among the variables is high (violating the
compatibility condition in the Example above), one
can hardly expect the Lasso or any other variable selection method to
satisfy (A2) for very small $\delta$. In
Section \ref{sectionempirical} we present empirical results showing that
the hierarchical p-value method still works well even when the screening
property is not satisfied for a small value $\delta$, and we provide some
supporting theoretical results for this fact in Section
\ref{sectionrobust}. 

For a given hierarchy $\mathcal{T}$, denote the set of clusters
that fulfill the null hypothesis by
$$\mathcal{T}_0 := \{ C \in \mathcal{T}\,:\,H_{0,C} \mbox{ is
  fulfilled}\}.$$
Furthermore, for some fixed parameter $\gamma \in (0,1)$ and some fixed
significance level $\alpha \in (0,1)$,
$$\mathcal{T}^\gamma_{rej} = \{ C \in \mathcal{T}\,:\, Q_h^C(\gamma)
\leq \alpha \}$$ 
is the set of rejected clusters based on the
p-values $Q_h^C(\gamma)$ and analogously,
$$\mathcal{T}_{rej} = \{ C \in \mathcal{T}\,:\, P_h^C
\leq \alpha \}$$ is the set of the rejected clusters when considering the
p-values $P_h^C$. The latter does not require to choose or pre-specify a
parameter like $\gamma$. 
\begin{theo}\label{theo1}
Assume that (A1) and (A2) hold. Then for any significance level $\alpha \in
(0,1)$ and $B$ denoting the number of sample splits:
\begin{enumerate}
\item For any fixed $\gamma \in
(0,1)$, the p-values $Q_h^C(\gamma)$ control the familywise error rate in the
sense that: 
$$\PP(\mathcal{T}^\gamma_{rej} \cap \mathcal{T}_0 \neq \emptyset) \leq
\alpha + 1 - (1-\delta)^B \leq \alpha + B\delta.$$
\item The p-values $P_h^C$ control the familywise error rate in the sense that:
$$\PP(\mathcal{T}_{rej} \cap \mathcal{T}_0 \neq \emptyset) \leq \alpha+ 1 -
(1-\delta)^B \leq \alpha + B\delta.$$
\end{enumerate}
\end{theo}
A proof is given in the Supplemental Material.
From Theorem \ref{theo1}, 
providing non-asymptotic bounds for familywise error rate control, one can
easily derive asymptotic familywise error control using the assumption  
\begin{eqnarray*}
&&\mbox{(A2') \it{Screening property}:}\; \lim_{n \to \infty} \PP[\hat{S}
\supseteq S_0] =1.
\end{eqnarray*}

\medskip\noindent
\textbf{Example (continued).} For the Lasso, under the assumption 1.
described in the Example above, and replacing assumption 2. by an
asymptotic beta-min condition 
\begin{eqnarray}\label{beta-min}
\min_{j \in S_0}|\beta^0_j| \gg \sqrt{\frac{\log(p)}{n}} s_0/\phi_0^2,
\end{eqnarray}
we have that (A2') holds (as $n \to \infty$, $p = p_n$ and $s_0 =
s_{0;n}$ and $\phi_0^2 = \phi_{0;n}^2$ are allowed to change with $n$). 

We then have the following result. 
\begin{corr}\label{cor1} 
Assume that (A1) and (A2') hold. Then for any fixed $\gamma \in (0,1)$
and significance level $\alpha \in (0,1)$:
$$\limsup_{n\to\infty} \PP(\mathcal{T}^\gamma_{rej} \cap \mathcal{T}_0 \neq
\emptyset) \leq \alpha$$
$$\limsup_{n\to\infty} \PP(\mathcal{T}_{rej} \cap \mathcal{T}_0 \neq
\emptyset) \leq \alpha.$$
\end{corr}
%

\subsection{Shaffer improvement in high-dimensional setting}\label{sectionShaffer}
A similar version of the Shaffer improvement as described in 
\citet[Section 2.4]{Meins08} can be applied to our method. 
The main idea, shown by \citet{Shaf86}, is that in a hierarchical
structure, some combinations of null hypothesis can be excluded a priori,
and incorporating constraints on the possible combinations of null
hypotheses can increase the power of the method. 

Consider a binary hierarchy $\mathcal{T}$ and a screened set $\hat{S} \subset
\{1, \dots, p\}$. The siblings of a cluster $C$ are the children of the
parent of $C$ which are not identical to $C$, $\si(C)=\ch(\pa(C))
\backslash C$. Define the effective cluster size $|C|^{\hat{S}}_{e\!f\!f}$ of
the cluster $C \in \mathcal{T}$ restricted to the screened set $\hat{S}$ as
\begin{eqnarray*}
|C|^{\hat{S}}_{e\!f\!f}= \left\{
\begin{array}{ll}
|C \cap \hat{S}|,&\mbox{if $\exists\, E \in \ch(\si(C))$
  s.t. $E \cap
  \hat{S} \neq \emptyset$}\\
|C \cap \hat{S}|+|\si(C) \cap \hat{S}|,&\mbox{otherwise.}
\end{array}
\right.
\end{eqnarray*}
Note that when no screening is performed ($\hat{S} = \{1, \dots, p\}$)
this definition coincides with the definition of the effective cluster size
in \citet{Meins08}. Moreover the condition ``$\exists\, E \in \ch(\si(C))$
  s.t. $E \cap \hat{S} \neq \emptyset$'' is stronger than the condition
  ''$\si(C)$ is not a leaf node`` of \citet{Meins08} and hence the
  improvement given by our definition of restricted cluster size is bigger
  than the one given by a straightforward adaption of \citet{Meins08}.

The Shaffer improvement in the high dimensional setting is then given
by considering the multiplicity adjustment  
\begin{equation}\label{Shaffermultiplicity}
p_{adj}^{C,(b)} = \min \big(\, p^{C,(b)}
\frac{|\hat{S}^{(b)}|}{|C|^{\hat{S}^{(b)}}_{e\!f\!f}}~,~1 \big),
\end{equation}
instead of using the multiplicity adjustment in (\ref{multiplicityadj}).

Obviously, since the effective cluster size is always at least as big as the
cluster size, the Shaffer improvement produces smaller p-values and hence
increases the power of the method while
the familywise error rate control is still guaranteed, as described next. 
\begin{theo}\label{theoShaffer}
Assume the hierarchy $\mathcal{T}$ is binary. Then, Theorem \ref{theo1} and Corollary \ref{cor1} still hold when using the
Shaffer improvement (\ref{Shaffermultiplicity}) as multiplicity adjustment,
assuming the conditions of Theorem \ref{theo1} or Corollary \ref{cor1},
respectively. 
\end{theo}
A proof is given in the Supplemental Material. We note that an extension of
the results in Theorem \ref{theo1} and \ref{theoShaffer} to control the
false discovery rate \citep{benjamini1995controlling}, instead of the FWER,
for \emph{hierarchically} ordered hypotheses (with corresponding dependent
p-values) seems very challenging.

\section{Empirical results}\label{sectionempirical}
In this section we study the performance of our hierarchical method and
compare it with the single-variable testing method of \citet{memepb09}. Section
\ref{sectionBigblocks} provides most informative results about our new
method, in particular for understanding the differences in comparison to
the single-variable approach.  

In a simulation study, we consider both synthetic and semi-real 
data. The former are used to study special designs where
we expect one of the two methods to perform clearly better. The semi-real
data are used to obtain insights of what happens when the design matrix comes
from real high-dimensional datasets. In our simulation study, all the data
are generated from a linear model
\begin{eqnarray*}
Y = \bx \beta^0 + \eps,
\end{eqnarray*}
where $\bx$ is a $n \times p$ matrix from synthetic (designs 1 to 3) or
real (designs 4 to 7) data, $\beta^0$ is a $p \times 1$ synthetic
regression vector and $\eps \sim \mathcal{N}_n(0,\sigma^2I_n)$ is a
synthetic noise term. The data are always standardized such that $\bx$ has
columns with empirical mean zero and variance one. 

We also apply the two different methods to a real dataset in Section
\ref{sectionMotif}.  

\subsection{Implementation of the method}\label{sectionmethods}

The implementation of our hierarchical method, and also of the single
variable procedure from
\citet{memepb09}, requires to make some choices. We mostly consider fairly
standard and ``easy to use'' methods; unless there is some deeper
methodological difference, as in our choice of additionally considering a
less standard clustering procedure. 

For clustering we consider the recently proposed canonical correlation
clustering of \citet{buru12} and the standard hierarchical clustering 
(using the \texttt{R}-fuction \texttt{hclust})
with distance between two covariables set as 1 less the absolute
correlation between the covariables, using complete linkage
(other linkages lead to similar results). For variable screening, we
use the Lasso (i.e., $\hat{S}$ from the non-zero estimated coefficients
from Lasso) with regularization parameter chosen by 10-fold
cross-validation. 

As in \citet{memepb09}, we choose $B=50$ as the number of
sample splits.
For aggregation, the p-values $P_h^C$ in (\ref{defiPC}) are computed over a
grid of 
$\gamma$-values between $\gamma_{min}=0.05$ and $1$ with grid-steps of size
$0.025$. For both hierarchical methods we use the Shaffer improvement
described in 
Section \ref{sectionShaffer}. As nominal significance level we always
consider $\alpha=5\%$. 

\subsection{Simulation study with synthetic and semi-real data}\label{sectionScenarios}
We consider 42 scenarios based on 7 designs. For each design we
consider 6 settings by varying the number of variables $p$ in the model
($p=200$, $p=500$ and $p=1000$) and the signal to noise ratio (SNR, for
each design and choice of $p$ we consider a low and a high SNR, namely for
$p=200$ we use $\mbox{SNR}=4$ and $\mbox{SNR}=8$, for $p=500$ we use
$\mbox{SNR}=8$ and $\mbox{SNR}=16$, for $p=1000$ we use $\mbox{SNR}=16$ and 
$\mbox{SNR}=32$). The signal to noise ratio is defined by
\begin{eqnarray*}
\mbox{SNR} = \sqrt{\frac{(\beta^0)^T\bx^T\bx\beta^0}{n\sigma^2}}
\end{eqnarray*}
and our choices of signal to noise ratios are avoiding scenarios where the
methods have degenerate performance of 0\% or 100\%, respectively. 
In designs 1 to 5 the sparsity $s_0$
is set to be 10, while in designs 6 and 7 it is set to be 6. The non-zero 
components of $\beta^0$ are randomly set as $\beta_j^0=1$ or
$\beta_j^0=-1$ for $j \in S_0$. The choice of $S_0$ 
is design-specific and hence explained with
the descriptions of the designs as follows.\\

{\bf Design 1: equi correlation}\\
We set $n=100$ and generate $\bx$ from a centered multivariate normal
distribution with equal variances $\rho_{jj}=1$ and covariances equal to 
$\rho_{jk}=0.3$ between variables $j$ and $k$ for $j \neq k \in
\{1,\dots,p\}$. The 10 active variables are chosen randomly among the $p$ 
covariables.\\

{\bf Design 2: high correlation within small blocks}\\
We set $n=100$ and generate $\bx$ from a centered multivariate normal
distribution with covariance $\rho_{jk}$ between variables $j$ and $k$ set
as $\rho_{jj}=1$ for all $j$, $\rho_{j,j+1}=\rho_{j+1,j}=0.9$ for $j \in
\{1,3,5,7,9,11,13,15,17,19\}$ and $\rho_{jk}=0$ otherwise. We choose in each
of the 10 two-dimensional blocks with high correlation one active variable,
i.e. for $j \in \{1,3,5,7,9,11,13,15,17,19\}$ we choose randomly either $j \in
S_0$ or $j+1 \in S_0$.\\

{\bf Design 3: high correlation within large blocks}\\
We set $n=100$ and generate $\bx$ from a centered multivariate normal
distribution with a block diagonal covariance matrix with 10
$p/10$-dimensional blocks $B_{p/10}(0.9)$ defined by
$(B_{p/10}(0.9))_{jj}=1$ and $(B_{p/10}(0.9))_{jk}=0.9$ for $j \neq k$. We
randomly choose in each of these p/10-dimensional blocks with high
correlation one active variable.\\

{\bf Design 4: Riboflavin dataset with normal correlation}\\
We consider the Riboflavin dataset \citep{bumeka13} with $n=71$
and choose randomly $p$ (i.e. 200, 500 or 1000 depending on the setting)
among 4088 covariables in the whole dataset. 
The 6 active variables are chosen randomly among the $p$ covariables.\\

{\bf Design 5: Breast dataset with normal correlation}\\
We consider the Breast dataset \citep{vV2002} with $n=117$
and choose randomly $p$ (i.e. 200, 500 or 1000 depending on the setting)
among 24481 covariables in the whole dataset. 
The 10 active variables are chosen randomly among the $p$ covariables.\\

{\bf Design 6: Riboflavin dataset with high correlation}\\
We consider again the Riboflavin dataset as in
design 4, but choose $p$ covariables as follows: a covariable is randomly
chosen among all 4088 covariables in the whole dataset. Then the 9
covariables with the highest absolute correlation with the first one are
chosen to build an ''high correlated'' 10-dimensional block. Then another
covariable is chosen among the remaining 4078 covariables of the whole
dataset and another ''high correlated'' 10-dimensional block is analogously
built. We repeat this procedure until we have $p$ covariables. The 6 active
variables are chosen randomly among the set $\{j; j=10k+1, 0 \leq k \leq
p/10-1 \}$.\\

{\bf Design 7: Breast dataset with high correlation}\\
We consider again the Breast dataset as in
design 5, but choose $p$ covariables as follows: a covariable is randomly
chosen among all 24481 covariables in the whole dataset. Then the 9
covariables with the highest absolute correlation with the first one are
chosen to build an ''high correlated'' 10-dimensional block. Then another
covariable is chosen among the remaining 24471 covariables of the whole
dataset and another ''high correlated'' 10-dimensional block is analogously
built. We repeat this procedure until we have $p$ covariables. The 10 active
variables are chosen randomly among the set $\{j; j=10k+1, 0 \leq k \leq
p/10-1 \}$.\\

\subsubsection{Performance measures for simulation
  study}\label{sectionevaluation} 

Besides the familywise error rate, we consider, among other aspects, the
following one-dimensional statistics measuring power (while
Section \ref{sectionBigblocks} provides
a more informative picture by avoiding to compress to one-dimensional
performance measures).  
 
We use two different performance functions. The first one
is defined as
\begin{eqnarray}\label{performanceFunction1}
\mbox{Performance 1} = \frac{1}{|S_0|} \sum_{\mbox{MTD C}} \frac{1}{|C|},
\end{eqnarray}
where the sum is over all minimal true detections (which we denote
by ``MTD''). Thereby:
\begin{changemargin}{0.12\textwidth}{0.12\textwidth} 
A cluster is said to be a MTD if it satisfies all of the following:
\begin{itemize}
\item $C$ is a significant cluster, e.g. has p-value $<5\%$. (``Detection'')
\item There is no significant sub-cluster $D \subset C$. (``Minimal'')
\item $C \notin \mathcal{T}_0$, i.e. there is at least one active variable in
$C$. (``True'')
\end{itemize}
\end{changemargin}
The Performance 1 is always between 0 and 1, and it is exactly 1 when
each active variable is selected as a singleton. Moreover, the contribution 
to the Performance 1 of MTD $C$ is independent from the number of active
variables that are in $C$. Although 
this penalizes our new method, it reflects the fact that from $P^C_h<5\%$
one can only conclude that there is at least one active variable in $C$
without having further information whether there are additional active
variables in $C$ and which of the variables in $C$ are active.

As second performance function we consider a slightly modified version of
the Performance 1, where only MTDs with cardinality $|C| \leq 20$
are considered, and a ``bonus'' is given for each MTD,
independently from its cardinality (if the latter is at most 20): 
\begin{eqnarray}\label{performanceFunction2}
\mbox{Performance 2} = \frac{1}{|S_0|} \sum_{\mbox{MTD C with }|C| \leq
20} \frac{1}{2} \Big( \frac{1}{|C|}+1 \Big).
\end{eqnarray}
The Performance 2 is also always between 0 and 1, and it is again exactly
equal to 1 if each active variable is selected as a singleton. 

Moreover, for the single variable method, both performance measures are the
same as only singletons can be selected. Correct selection of a cluster
with more than one variable is less valuable than a singleton, with both
performance measures: Performance 2, however, is putting less emphasis on
the size of a selected cluster. The choice of the bound for the cluster
being at most 20 in Performance 2 is motivated by the idea that too large
clusters are ``uninteresting'' in many practical applications (e.g. a
genetic pathway consists of about up to 20 genes, and a cluster would
represent a pathway). 

\subsubsection{Familywise error rate control (FWER)}
For each of the 42 scenarios described in Section \ref{sectionScenarios} 
we make 100 independent simulation runs varying only
the synthetic noise term $\eps$ and count the number where at least
one false selection is made (i.e. there exists a cluster $C \in \mathcal{T}_0
\cap \mathcal{T}_{rej}$). According to Theorem \ref{theoShaffer} we
expect this number to be at most $100 \alpha = 5$ ($\alpha = 0.05$). 
\begin{table}[!h]
\centering
\begin{tabular}[h]{|l|c||c|c|c||c|c|c|}
\hline
& & \multicolumn{6}{c|}{Familywise error rate (in \%)}\\
\cline{3-8}
Design & $p$ & \multicolumn{3}{c||}{low SNR} & \multicolumn{3}{c|}{high
  SNR}\\
\cline{3-8}
& & Single & Cancorr & Hclus & Single & Cancorr & Hclus\\
\hline\hline
& 200 & 0 & 0 & 0 & 0 & 0 & 0\\
\cline{2-8}
equi & 500 & 0 & 0 & 0 & 0 & 0 & 0\\
\cline{2-8} 
corr & 1000 & 0 & 0 & 0 & 0 & 0 & 0\\
\hline
& 200 & 5 & 5 & 5 & 0 & 0 & 0\\
\cline{2-8}
small & 500 & \cellcolor{first}7 & \cellcolor{first}7 & \cellcolor{first}8 & 0 & 0 & 0\\
\cline{2-8} 
blocks & 1000 & 0 & 0 & 0 & \cellcolor{first}7 & \cellcolor{first}8 & \cellcolor{first}8\\
\hline
& 200 & 0 & \cellcolor{first}18 & \cellcolor{first}8 & 0 & 0 & 0\\
\cline{2-8}
big & 500 & 0 & 0 & 0 & 0 & 0 & 0\\
\cline{2-8} 
blocks & 1000 & 0 & 0 & 0 & 0 & 0 & 0\\
\hline
Riboflavin & 200 & 0 & 0 & 0 & 0 & 0 & 0\\
\cline{2-8}
normal & 500 & 0 & 0 & 0 & 0 & 0 & 0\\
\cline{2-8} 
corr & 1000 & 0 & 0 & 0 & 0 & 0 & 0\\
\hline
Breast & 200 & 0 & 0 & 0 & 0 & 0 & 0\\
\cline{2-8}
normal & 500 & 0 & 0 & 0 & 0 & 0 & 0\\
\cline{2-8} 
corr & 1000 & 0 & 0 & 0 & 0 & 0 & 0\\
\hline
Riboflavin & 200 & 0 & 0 & 0 & 0 & 0 & 0\\
\cline{2-8}
high & 500 & 0 & 0 & 0 & 0 & 0 & 0\\
\cline{2-8} 
corr & 1000 & 0 & 0 & 0 & 0 & 0 & 0\\
\hline
Breast & 200 & 0 & 0 & 0 & 0 & 0 & 0\\
\cline{2-8}
high & 500 & 0 & 0 & 0 & 0 & 0 & 0\\
\cline{2-8} 
corr & 1000 & 0 & 0 & 0 & 1 & 1 & 2\\
\hline
\end{tabular}
\caption{Familywise error rate in \%: Number of cases with at least one false
  selection, out of 100 simulation runs. The scenarios where the critical
  value of 5 is overtaken are marked in gray (only one scenario with 18\%
  is substantially differing from the nominal 5\% level).}
\label{tableFWER}
\end{table}
The results illustrated in Table \ref{tableFWER} show that for 39 of the 42
scenarios, FWER control 
holds for all methods, while it doesn't hold for any method in two scenarios 
and for the hierarchical methods in one scenario. In
37 out of the 42 scenarios there is no false selection at all. It is not
surprising that the most problematic design with respect to FWER is the
``small blocks'' design, since there each active predictor is
highly correlated with a false variable from $S_0^c$ and hence it is rather
difficult for our screening method (the Lasso) to guarantee $\hat{S}
\supseteq S_0$. 

\subsubsection{Power: Performance 1}
For each of the 42 scenarios described in Section \ref{sectionScenarios} we
make 100 simulation runs varying the synthetic noise term $\eps$ and the synthetic
regression vector $\beta^0$. We then calculate the average Performance 1 in 
(\ref{performanceFunction1}), i.e., Performance 1 is averaged over 100
simulation runs.  
\begin{table}[!h]
\centering
\begin{tabular}[h]{|l|c||c|c|c||c|c|c|}
\hline
& & \multicolumn{6}{c|}{Performance 1 in \%}\\
\cline{3-8}
Design & $p$ & \multicolumn{3}{c||}{low SNR} & \multicolumn{3}{c|}{high
  SNR}\\
\cline{3-8}
& & Single & Cancorr & Hclus & Single & Cancorr & Hclus\\
\hline\hline
& 200 & \cellcolor{first}48.9 & \cellcolor{second}48.8 & 45.6 & \cellcolor{first}97.5 & 97.1 & \cellcolor{second}97.3\\
\cline{2-8}
equi & 500 & \cellcolor{first}45.1 & \cellcolor{second}44.6 & 42.9 & \cellcolor{second}78.0 & \cellcolor{first}78.0 & 76.7\\
\cline{2-8} 
corr & 1000 & \cellcolor{first}23.6 & \cellcolor{second}23.3 & 22.2 & \cellcolor{first}26.9 & \cellcolor{second}26.5 & 25.7\\
\hline
& 200 & 29.5 & \cellcolor{second}39.8 & \cellcolor{first}40.7 & 89.1 & \cellcolor{first}95.1 & \cellcolor{second}94.7\\
\cline{2-8}
small & 500 & 37.2 & \cellcolor{second}45.1 & \cellcolor{first}45.6 & 56.2 & \cellcolor{second}60.1 & \cellcolor{first}60.6\\
\cline{2-8} 
blocks & 1000 & 14.2 & \cellcolor{second}16.9 & \cellcolor{first}17.2 & 20.0 & \cellcolor{second}21.6 & \cellcolor{first}21.5\\
\hline
& 200 & 2.9 & \cellcolor{second}7.4 & \cellcolor{first}7.7 & 18.2 & \cellcolor{second}24.7 & \cellcolor{first}24.8\\
\cline{2-8}
big & 500 & 0.0 & \cellcolor{second}0.1 & \cellcolor{first}0.3 & 6.3 & \cellcolor{second}6.8 & \cellcolor{first}7.4\\
\cline{2-8} 
blocks & 1000 & 0.0 & 0.0 & 0.0 & \cellcolor{second}5.0 & 4.9 & \cellcolor{first}5.1\\
\hline
Riboflavin & 200 & \cellcolor{first}19.7 & \cellcolor{second}19.0 & 19.0 & \cellcolor{second}46.0 & 45.7 & \cellcolor{first}46.5\\
\cline{2-8}
normal & 500 & \cellcolor{first}15.7 & \cellcolor{second}15.6 & 15.2 & \cellcolor{first}33.2 & \cellcolor{second}33.0 & 31.7\\
\cline{2-8} 
corr & 1000 & \cellcolor{second}10.7 & 10.5 & \cellcolor{first}10.8 & \cellcolor{second}24.7 & \cellcolor{first}24.9 & 24.3\\
\hline
Breast & 200 & \cellcolor{second}38.8 & \cellcolor{first}39.6 & 38.3 & \cellcolor{second}90.2 & \cellcolor{first}90.4 & 89.8\\
\cline{2-8}
normal & 500 & \cellcolor{first}38.8 & \cellcolor{second}38.8 & 37.1 & \cellcolor{second}76.0 & \cellcolor{first}75.9 & 75.8\\
\cline{2-8} 
corr & 1000 & \cellcolor{second}33.9 & \cellcolor{first}33.9 & 31.7 & \cellcolor{first}43.6 & \cellcolor{second}43.4 & 42.4\\
\hline
Riboflavin & 200 & 25.7 & \cellcolor{second}25.9 & \cellcolor{first}26.7 & 58.8 & \cellcolor{second}59.2 & \cellcolor{first}59.6\\
\cline{2-8}
high & 500 & \cellcolor{second}37.2 & 36.9 & \cellcolor{first}37.5 & 61.8 & \cellcolor{second}62.0 & \cellcolor{first}62.2\\
\cline{2-8} 
corr & 1000 & \cellcolor{second}31.8 & 31.5 & \cellcolor{first}32.6 & \cellcolor{second}48.3 & 48.2 & \cellcolor{first}48.8\\
\hline
Breast & 200 & 42.4 & \cellcolor{second}43.0 & \cellcolor{first}44.4 & 84.4 & \cellcolor{first}85.3 & \cellcolor{second}85.0\\
\cline{2-8}
high & 500 & 55.3 & \cellcolor{second}55.4 & \cellcolor{first}55.5 & \cellcolor{second}89.8 & 89.6 & \cellcolor{first}90.1\\
\cline{2-8} 
corr & 1000 & 57.0 & \cellcolor{second}57.1 & \cellcolor{first}57.3 &
\cellcolor{second}72.8 & 72.7 & \cellcolor{first}73.1\\
\hline\hline
\multicolumn{2}{|l||}{Avg. normal corr.} & \cellcolor{first}30.6 & \cellcolor{second}30.4 & 29.2 & \cellcolor{first}57.3 & \cellcolor{second}57.2 & 56.7\\
\hline
\hline\hline
\multicolumn{2}{|l||}{Avg. high corr} & 27.8 & \cellcolor{second}29.9 & \cellcolor{first}30.5 & 50.9 & \cellcolor{second}52.5 & \cellcolor{first}52.7\\
\hline
\hline\hline
\multicolumn{2}{|l||}{Average} & 29.0 & \cellcolor{first}30.2 & \cellcolor{second}29.9 & 53.7 & \cellcolor{first}54.5 & \cellcolor{second}54.4\\
\hline
\end{tabular}
\caption{Performance 1, averaged over 100 simulation runs, for the 
  methods ``single variable'', ``hierarchical with canonical correlation
  clustering'' and  
``hierarchical with \texttt{hclust} clustering''. The best and
  second best methods are marked in dark-gray and light-gray. The average
  performances in the bottom rows are averages over the corresponding or all
scenarios, respectively.}
\label{tablePerformance1}
\end{table}

The results are reported in Table \ref{tablePerformance1}. They show that, as 
expected, the hierarchical methods provide better results in the designs where the
correlation among the variables is rather high (designs 2,3,6 and 7), while
in the other designs the non-hierarchical method has in general a slightly
better performance. 
The method based on the \texttt{hclust} clustering is 
more sensitive with respect to high correlation among the variables than
the analogue based on canonical correlation clustering. In particular the
latter is best in 21 of the 24 scenarios which use design 2,3,6 and 7 while
it is the worst method in 16 of the 18 scenarios where the correlation is 
not particularly high.  
We note that the differences among the methods are rather small: this is
mainly a consequence of our definition 
(\ref{performanceFunction1}) of the Performance 1 and $p$ being large. The
biggest (absolute) difference in the Performance 1 can be found in 
design 2 where the hierarchical methods have a performance up to 11.2
percent higher than the single variable method, while the biggest
deficit of a hierarchical method with respect to the single variable method
can be found in design 1 and amounts to 3.3 percent. As expected,
our results show that in general the Performance 1 (and also the differences
between them when considering the different methods) lowers when $p$
increases. Finally, it is interesting to note that in the scenarios that
favor the single variable method, the Performance 1 of the method with
canonical correlation clustering is very close to the Performance 1 of the
single variable method (the difference is at most 0.8 percent),
while in the other scenarios it might perform much better (differences of
up to 10.3 percent). 

\subsubsection{Power: Performance 2}

In Table \ref{tablePerformance2} we show the average Performance 2 of the
three considered methods for the 42 different scenarios, i.e., for each
scenario, Performance 2 is averaged over 100 simulation runs. While by
definition, Performance 1 and 2 are the same for the single variable
method, we find for both hierarchical methods that the
Performance 2 is generally higher than the Performance 1 (only in 6 out of
84 cases it is lower and the difference is at most 0.1 percent). This was
expected as the idea of Performance 2 is to give a little extra
reward to each correct selection, independently from the cardinality of the
selected cluster (given the latter is at most 20). In particular, the
method that benefits most from Performance 2 is the hierarchical method with
\texttt{hclust} clustering which has an average Performance 2 of 45.2 percent while
its average Performance 1 is 42.1 percent
(average is meant over all 
scenarios). We also note that for Performance 2, the difference between the
single variable and the hierarchical methods in the settings with high
correlation is much more evident. 
\begin{table}[!h]
\centering
\begin{tabular}[h]{|l|c||c|c|c||c|c|c|}
\hline
& & \multicolumn{6}{c|}{Performance 2 in \%}\\
\cline{3-8}
Design & $p$ & \multicolumn{3}{c||}{low SNR} & \multicolumn{3}{c|}{high
  SNR}\\
\cline{3-8}
& & Single & Cancorr & Hclus & Single & Cancorr & Hclus\\
\hline\hline
& 200 & \cellcolor{second}48.9 & \cellcolor{first}49.1 & 47.7 & \cellcolor{second}97.5 & 97.2 & \cellcolor{first}97.6\\
\cline{2-8}
equi & 500 & \cellcolor{first}45.1 & \cellcolor{second}44.6 & 43.5 & \cellcolor{second}78.0 & \cellcolor{first}78.0 & 77.0\\
\cline{2-8} 
corr & 1000 & \cellcolor{first}23.6 & \cellcolor{second}23.3 & 22.5 & \cellcolor{first}26.9 & \cellcolor{second}26.5 & 26.0\\
\hline
& 200 & 29.5 & \cellcolor{second}44.6 & \cellcolor{first}48.8 & 89.1 & \cellcolor{first}97.0 & \cellcolor{second}96.7\\
\cline{2-8}
small & 500 & 37.2 & \cellcolor{second}48.4 & \cellcolor{first}51.4 & 56.2 & \cellcolor{second}62.1 & \cellcolor{first}63.7\\
\cline{2-8} 
blocks & 1000 & 14.2 & \cellcolor{second}18.1 & \cellcolor{first}19.7 & 20.0 & \cellcolor{second}22.2 & \cellcolor{first}23.4\\
\hline
& 200 & 2.9 & \cellcolor{second}34.3 & \cellcolor{first}34.4 & 18.2 & \cellcolor{second}58.9 & \cellcolor{first}59.0\\
\cline{2-8}
big & 500 & 0.0 & \cellcolor{second}0.0 & \cellcolor{first}0.1 & 6.3 & \cellcolor{second}6.8 & \cellcolor{first}7.8\\
\cline{2-8} 
blocks & 1000 & 0.0 & 0.0 & 0.0 & \cellcolor{second}5.0 & 4.9 & \cellcolor{first}5.2\\
\hline
Riboflavin & 200 & \cellcolor{second}19.7 & 18.9 & \cellcolor{first}21.0 & \cellcolor{second}46.0 & 45.8 & \cellcolor{first}48.3\\
\cline{2-8}
normal & 500 & \cellcolor{second}15.7 & 15.5 & \cellcolor{first}15.9 & \cellcolor{first}33.2 & \cellcolor{second}33.0 & 32.4\\
\cline{2-8} 
corr & 1000 & \cellcolor{second}10.7 & 10.6 & \cellcolor{first}11.1 & \cellcolor{second}24.7 & \cellcolor{first}24.9 & 24.7\\
\hline
Breast & 200 & 38.8 & \cellcolor{second}39.5 & \cellcolor{first}40.2 & 90.2 & \cellcolor{first}90.5 & \cellcolor{second}90.5\\
\cline{2-8}
normal & 500 & \cellcolor{first}38.8 & \cellcolor{second}38.7 & 38.3 & \cellcolor{second}76.0 & 75.9 & \cellcolor{first}76.3\\
\cline{2-8} 
corr & 1000 & \cellcolor{second}33.9 & \cellcolor{first}33.9 & 32.1 & \cellcolor{first}43.6 & \cellcolor{second}43.4 & 42.8\\
\hline
Riboflavin & 200 & 25.7 & \cellcolor{second}25.8 & \cellcolor{first}32.2 & 58.8 & \cellcolor{second}59.2 & \cellcolor{first}63.0\\
\cline{2-8}
high & 500 & \cellcolor{second}37.2 & 36.8 & \cellcolor{first}40.4 & 61.8 & \cellcolor{second}62.0 & \cellcolor{first}63.5\\
\cline{2-8} 
corr & 1000 & \cellcolor{second}31.8 & 31.5 & \cellcolor{first}33.6 & \cellcolor{second}48.3 & 48.2 & \cellcolor{first}50.0\\
\hline
Breast & 200 & 42.4 & \cellcolor{second}44.2 & \cellcolor{first}49.5 & 84.4 & \cellcolor{first}86.7 & \cellcolor{second}88.0\\
\cline{2-8}
high & 500 & 55.3 & \cellcolor{second}55.8 & \cellcolor{first}58.5 & \cellcolor{second}89.8 & 89.7 & \cellcolor{first}90.9\\
\cline{2-8} 
corr & 1000 & 57.0 & \cellcolor{second}57.1 & \cellcolor{first}58.6 &
\cellcolor{second}72.8 & 72.7 & \cellcolor{first}73.7\\
\hline\hline
\multicolumn{2}{|l||}{Avg. normal corr.} & \cellcolor{first}30.6 & \cellcolor{second}30.5 & 30.2 & \cellcolor{first}57.3 & 57.2 & \cellcolor{second}57.3\\
\hline
\hline\hline
\multicolumn{2}{|l||}{Avg. high corr} & 27.8 & \cellcolor{second}33.1 & \cellcolor{first}35.6 & 50.9 & \cellcolor{second}55.5 & \cellcolor{first}57.1\\
\hline
\hline\hline
\multicolumn{2}{|l||}{Average} & 29.0 & \cellcolor{second}31.9 & \cellcolor{first}33.3 & 53.7 & \cellcolor{second}56.4 & \cellcolor{first}57.1\\
\hline
\end{tabular}
\caption{Performance 2, averaged over 100 simulation runs, for the methods
  ``single variable'', ``hierarchical with canonical correlation 
  clustering'' and  
``hierarchical with \texttt{hclust} clustering''. The best and
  second best methods are marked in dark-gray and light-gray. The average
  performances in the bottom rows are averages over the corresponding or all
scenarios, respectively.}
\label{tablePerformance2}
\end{table}
Some additional results regarding the variability of both
  Performance 1 and Performance 2 measures among the 100 different
  simulation runs is given in the Supplemental Material.

\subsection{A more detailed consideration}\label{sectionBigblocks} 


The power results in the previous section are given in terms of
one-dimensional performance functions. Here, we provide more information
what our new method actually does and how it performs when looking beyond
one-dimensional summary statistics. On the other hand, to keep the exposition at
reasonable length, we focus on fewer simulation scenarios only. 

We consider the ``small blocks''- and ``large blocks''-designs (designs 2
and 3 of Section \ref{sectionScenarios}) with $p=200$, $\mbox{SNR}=8$, $s_0=10$ with the
non-zero components of $\beta^0$ randomly set as $\beta^0_j=\pm 1$ and
various values for the nontrivial covariances: $\rho \in \{0, 0.4, 0.7,
0.8, 0.85, 0.9, 0.95, 0.99\}$. For each of the 16 scenarios we make 100
simulation runs varying the synthetic noise term $\eps$. As results we
consider the FWER (portion of runs with at least a false detection over all
100 runs) and, averaged over the 100 runs, the number of MTDs and the number
of MTDs of some given cardinality. The results as shown in Table
\ref{tableSNRhigh}.

\begin{sidewaystable}[!htp]
\begin{tabular}[h]{|c|c||c|c|c||c|c|c||c|c|c||c|c||c|c||c|c|}
\hline
 & & \multicolumn{3}{c||}{FWER} &
\multicolumn{3}{c||}{\# MTD} &
\multicolumn{9}{c|}{\# MTD for given cardinality}\\
\cline{9-17}
$\rho$ & $\delta$ & \multicolumn{3}{c||}{} & \multicolumn{3}{c||}{} & \multicolumn{3}{c||}{$|\cdot|=1$} &
 \multicolumn{2}{c||}{$|\cdot|=2$} &
 \multicolumn{2}{c||}{$3 \leq |\cdot| \leq 10$} & 
 \multicolumn{2}{c|}{$11 \leq |\cdot| \leq 20$}
\\
\cline{3-17}
 & & S & C & H & S & C & H & S & C & H & C & H & C & H & C &
 H\\
\hline\hline
\multicolumn{17}{|c|}{``small blocks''-design with high SNR}\\
0 & 0.02 & 0 & 0 & 0 & 10 & 10 & 10 & 10 & 10 & 10 & 0 & 0 & 0 & 0
& 0 & 0\\
\hline
0.4 & 0.08 & 0 & 0 & 0 & 10 & 10 & 10 & 10 & 10 & 10 & 0 &
0 & 0 & 0 & 0 & 0\\
\hline
0.7 & 0.06 & 0 & 0 & 0 & 9.97 & 10 & 10 & 9.97 & 9.97 & 9.97 & 0.03 & 0.03 &
0 & 0 & 0 & 0\\
\hline
0.8 & 0.10 & 0 & 0 & 0 & 9.75 & 9.99 & 9.99 & 9.75 & 9.78 & 9.78 & 0.21 &
0.21 & 0 & 0 & 0 & 0\\
\hline
0.85 & 0.45 & 0 & 0 & 0 & 9.22 & 9.86 & 9.89 & 9.22 & 9.32 & 9.34 & 0.53 &
0.52 & 0 & 0 & 0.01 & 0.03\\
\hline
0.9 & 0.19 & 0 & 0 & 0 & 9.77 & 10 & 10 & 9.77 & 9.81 & 9.82 & 0.19 &
0.18 & 0 & 0 & 0 & 0\\
\hline
0.95 & 0.20 & 0 & 0 & 0 & 9.46 & 10 & 10 & 9.46 & 9.50 & 9.50 & 0.50 & 0.50 &
0 & 0 & 0 & 0\\ 
\hline
0.99 & 0.90 & 0.99 & 0.99 & 0.99 & 7.40 & 7.85 & 7.88 & 7.40 & 7.52 & 7.54 &
0.33 & 0.34 & 0 & 0 & 0 & 0\\
\hline
\hline
\multicolumn{17}{|c|}{``large blocks''-design with high SNR}\\
0 & 0.14 & 0 & 0 & 0 & 10 & 10 & 10 & 10 & 10 & 10 & 0 & 0 & 0 & 0
& 0 & 0\\
\hline
0.4 & 0.20 & 0 & 0 & 0 & 9.92 & 10 & 10 & 9.92 & 9.92 & 9.92 & 0.01 &
0.01 & 0.04 & 0.02 & 0.02 & 0.05\\
\hline
0.7 & 0.13 & 0 & 0 & 0 & 6.99 & 8.02 & 9.91 & 6.99 & 7.11 & 7.02 & 0 &
0.22 & 0.21 & 1.01 & 0.70 & 1.64\\
\hline
0.8 & 0.45 & 0 & 0 & 0 & 2.57 & 9.3 & 9.37 & 2.57 & 2.65 & 2.49 & 0.02 &
0.18 & 1.22 & 1.26 & 5.35 & 5.31\\
\hline
0.85 & 0.22 & 0 & 0 & 0 & 2.70 & 9.63 & 9.69 & 2.70 & 2.91 & 2.74 & 0 & 0.13 &
0.46 & 1.15 & 6.23 & 5.58\\
\hline
0.9 & 0.37 & 0 & 0 & 0 & 1.82 & 9.81 & 9.85 & 1.82 & 1.86 & 1.89 & 0.04 &
0.06 & 0.76 & 1.04 & 7.14 & 6.81\\
\hline
0.95 & 0.33 & 0 & 0 & 0.08 & 1.91 & 10 & 9.92 & 1.91 & 1.92 & 1.95 & 0.07 &
0.30 & 2.69 & 2.87 & 5.32 & 4.80\\ 
\hline
0.99 & 1.00 & 0.54 & 1.00 & 1.00 & 1.26 & 6.26 & 7.89 & 1.26 & 1.26 & 1.26 &
0.04 & 0.29 & 2.15 & 4.17 & 2.81 & 2.17\\
\hline 
\end{tabular}
\caption{Results of the simulation with the ``small
  blocks''- and ``large blocks''-design with high SNR (SNR=8) for different
  correlations.  
$\rho$ is the correlation in the design, $\delta$
the relative frequency of screenings with $\hat{S} \not\supset S_0$, MTD
denotes ``minimal true detections'', ``$2 \leq |\cdot| \leq 5$'' indicates
that MTD of cardinality between 2 and 5 are considered, S, C and H
represent the ``single variable'' resp. ''canonical correlation
clustering`` and ``hierarchical with \texttt{hclust} clustering'' method.}
\label{tableSNRhigh}
\end{sidewaystable}

FWER control (with nominal level $\alpha=5\%$) holds for most settings, even
for relatively high values of the fraction of failed screenings with
$\hat{S} \not\supset S_0$ (represented by $\delta$ in Theorem
\ref{theo1}). This indicates 
a robustness property of the methods in controlling FWER, beyond the
results of Theorem \ref{theo1} which requires that $\delta$ is very
small. 

Looking at the number of MTDs in Table \ref{tableSNRhigh}, we see that the
hierarchical method dominates the single variable method, with its
superiority increasing with
increasing correlation among the variables. Considering only the singleton
detections (MTDs with cardinality 1), there is only one scenario out of 16
where one of the two hierarchical methods is (slightly) worse than the
single variable method, while it is equal or (slightly) better for all
other settings.  

 We note that the hierarchical method can be better than the single
variable method because the Shaffer improvement of section
\ref{sectionShaffer} allows for better multiplicity adjustment. For the
scenarios with $\rho = 0$ (which by construction of the designs are the
same for ``small blocks'' and ``large blocks''), all 3 methods exhibit a
perfect accuracy. 

It is interesting to note that for the ``small blocks''-designs,
where the improvement given by the hierarchical over the single
variable method is smaller than for the ``large blocks''-designs, the
quality of the improvement should be considered as very 
high since almost all additional discoveries by the hierarchical method
have cardinality 2 only and often sum up to essentially all possible
discoveries. Further results regarding the MTDs for 4 (out of
the 16 considered) scenarios are given in the Supplemental Material.

\begin{figure*}[!htp]
\centerline{\includegraphics[width=1.1\textwidth, angle=0]{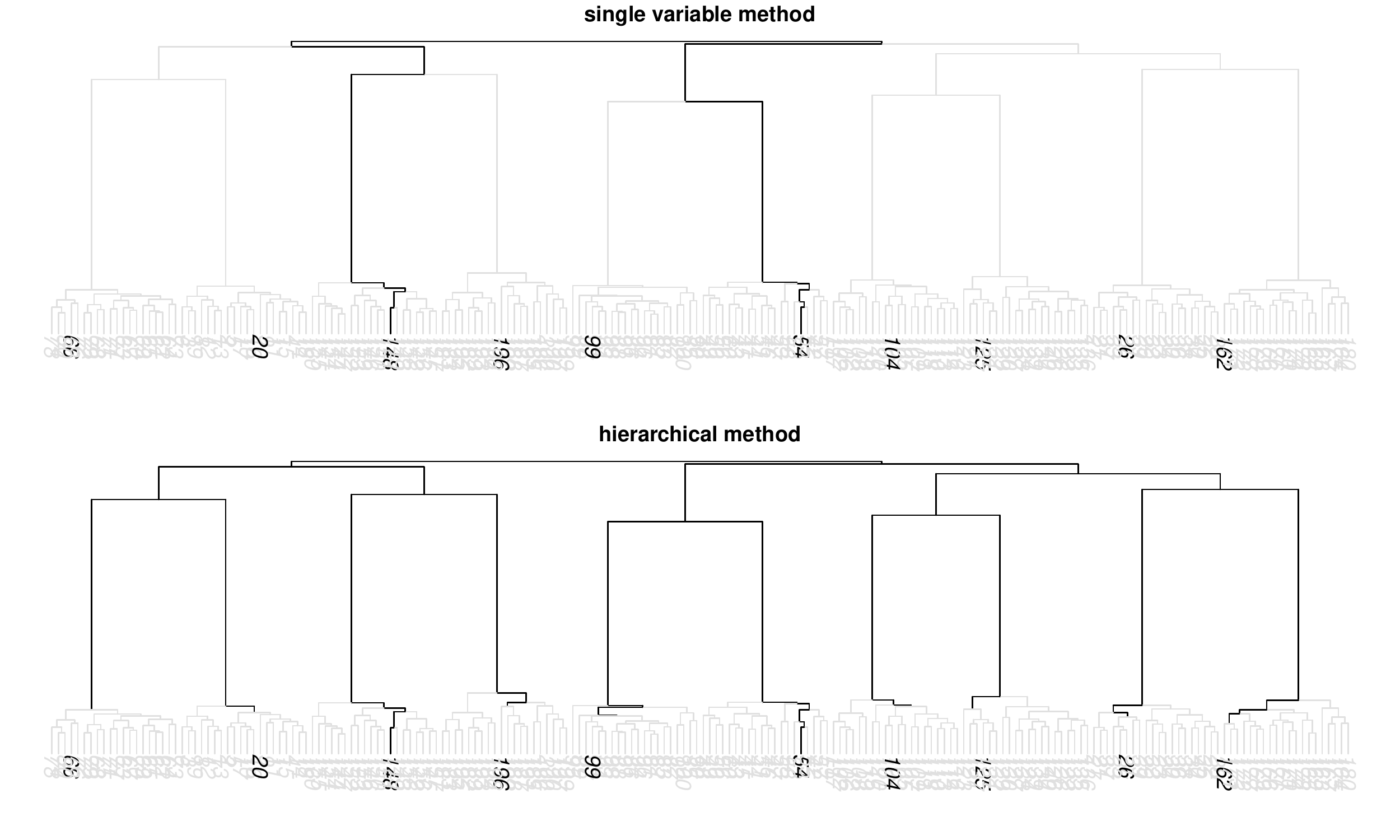}}
\vspace{-0.5cm}
\caption{Dendrograms for a representative run of the ``large
  blocks''-design with high SNR (SNR=8) and $\rho=0.85$. The active
  variables are labeled in black and the truly detected
  non-zero variables along the hierarchy are depicted in black.}
\label{figurebigblockshighSNR1}
\end{figure*}

For additional illustration, we
show in Figure \ref{figurebigblockshighSNR1} the dendrograms (in gray) for a
representative simulation run of the ``large blocks''-design with
$\rho=0.85$, for the
single variable method and the hierarchical method
with \texttt{hclust} clustering. The active variables are labeled in black
and truly detected non-zero variables
along the hierarchy are depicted in black. While the single 
variable method ``only'' detects 2 singletons, the hierarchical method
detects the same 2 singletons and achieves 8 more MTDs (one of which has
small cardinality 3 and hence is particularly informative). Figure
\ref{figuresmallblockshighSNR1} is analogous to Figure 
\ref{figurebigblockshighSNR1} for a simulation run of the ``small
blocks''-design with 
$\rho=0.8$. It shows that the hierarchical method improves the results of
the single variable method (9 detected singletons) by additionally
providing one MTD of cardinality 2 besides the same 9 singletons of the
single variable method. 
Thus, we provide evidence of the fact that the hierarchical method has the
powerful advantage of automatically going to the finer possible resolution,
depending on signal-strength and correlation structure among the variables.

\begin{figure*}[!htp]
\centerline{\includegraphics[width=1.1\textwidth, angle=0]{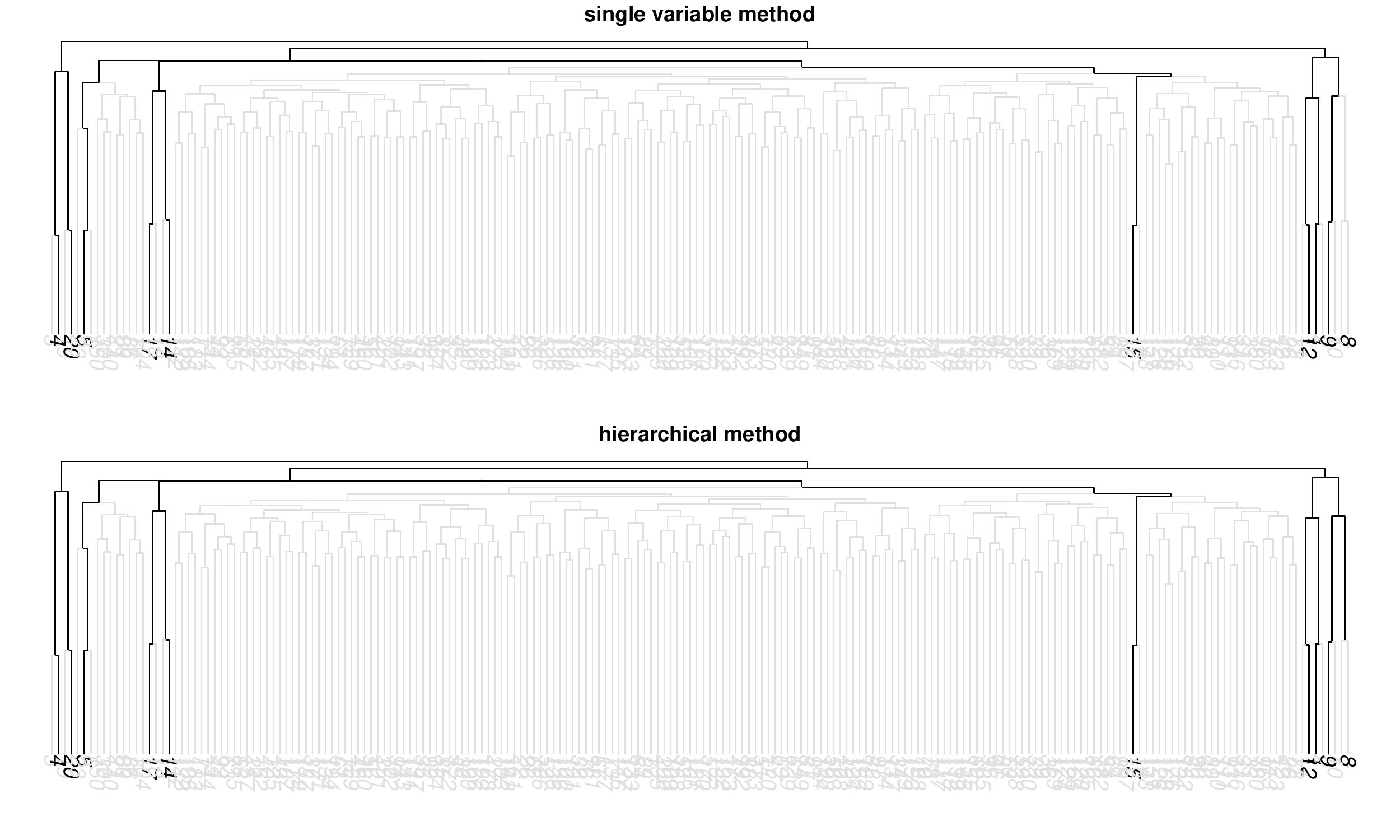}}
\vspace{-0.5cm}
\caption{Dendrograms for a representative run of the ``small
  blocks''-design with high SNR (SNR=8) and $\rho=0.80$. The active
  variables are labeled in black and the truly detected 
non-zero variables along the hierarchy are depicted in black.}
\label{figuresmallblockshighSNR1}
\end{figure*}

Finally, we illustrate in Figure \ref{figureTPRFPRhighSNR} the true
  positive (TPR) rates and false positive rates (FPR) of the Lasso, the
  single variable method and the hierarchical method with \texttt{hclust}
  clustering as points in the ROC space.

\begin{figure*}[!htp]
\centerline{\includegraphics[width=0.95\textwidth, angle=0]{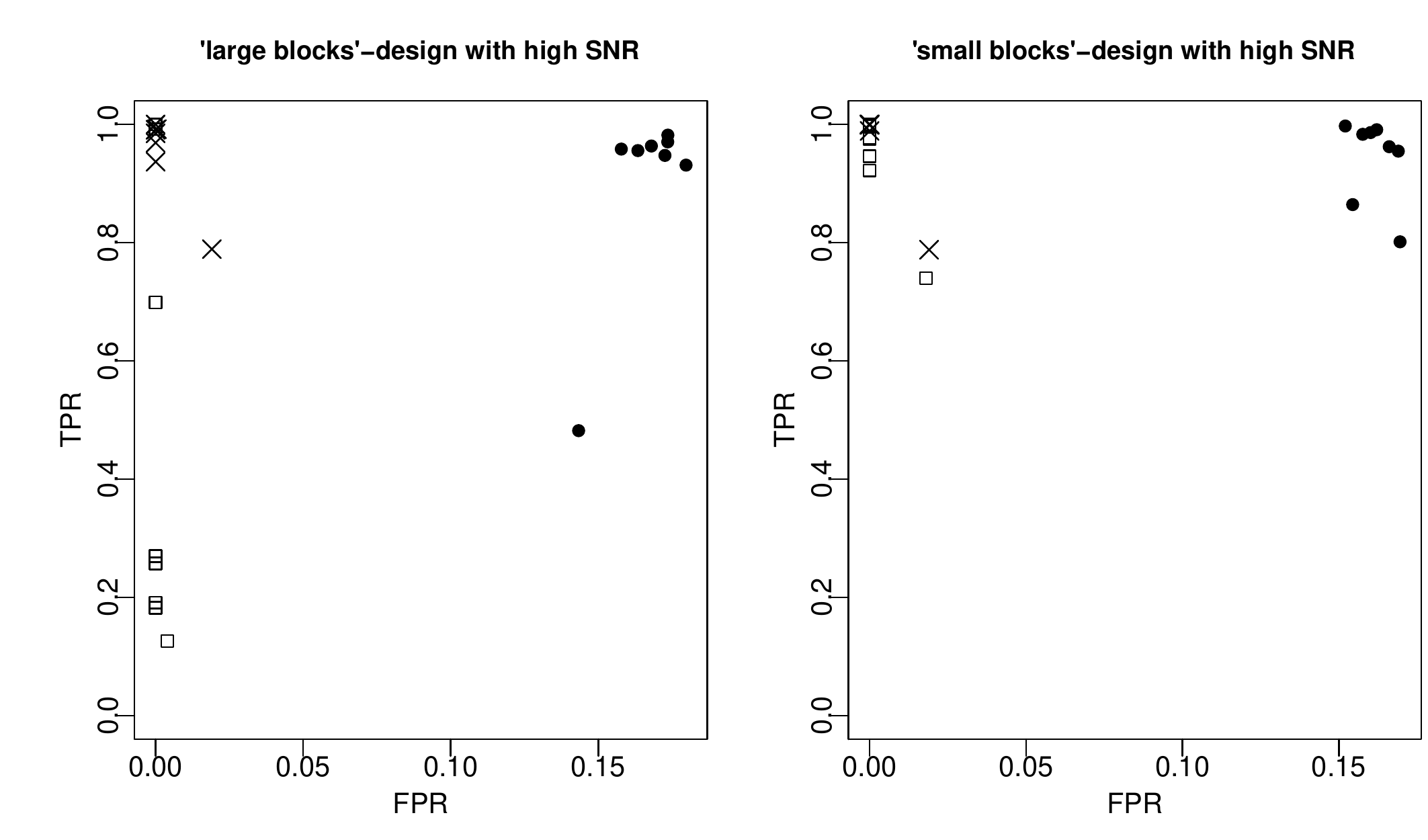}}
\vspace{-0.5cm}
\caption{True positive rate (TPR) and false positive rate (FPR) for the Lasso
  (bullet), the single variable method (box), and the hierarchical method with
  \texttt{hclust} clustering (cross) for different scenarios as indicated
  in the header of the plots.}
\label{figureTPRFPRhighSNR}
\end{figure*}

We note that, as expected from the philosophy of the single variable
  and hierarchical methods to control the FWER, there is a substantial
  difference between the FPR of the Lasso (0.15 to 0.18) and the FPR of the
  other 2 methods which is always less than 0.03 and equals 0 in most of
  the cases. For the ``small blocks''-design, this improvement of the FPR
  has no negative impact on the TPR, while for the more difficult ``large
  blocks''-design, a TPR comparable to that of the Lasso can only be
  achieved by the hierarchical method which significantly improves the TPR
  of the single 
variable method. It has to be remarked that the TPR and FPR are based on
MTDs (regardless from their cardinality), hence some care is needed when
comparing the TPR and FPR of the hierarchical with those of the other
methods where only singleton detections are possible. For a detailed
analysis we refer to Table \ref{tableSNRhigh}.

In the Supplemental Material we present the same detailed analysis as in this
section considering the same designs but with low $\mbox{SNR} = 4$ signal
to noise ratio. 

\subsection{Real data application: Motif Regression}\label{sectionMotif}
We apply the three methods described in Section \ref{sectionmethods} to a
real dataset about motif regression \citep{Conlon03} with $n = 287$ and
$p=195$, used in \citet[Section 4.3]{memepb09}.
The single variable method identifies one single 
predictor variable as significant (controlling the familywise error rate at
5\%). The same variable is found to be significant with the hierarchical
method with 
\texttt{hclust} clustering, while the hierarchical method with canonical
correlation clustering identifies as significant clusters, in the sense of
Section \ref{sectionevaluation}, the singleton, which is the same single
predictor as found by the other two methods, and a very big cluster of 165
variables. This is an interesting finding saying that besides the single
predictor variable, there are presumably other motifs, in the large
cluster, which play a relevant role. However, there is not enough
information to determine which of the variables in the large cluster are
significant as a single motif. 

\subsection{Conclusions from the empirical results}
We have studied error rate control and performance of the three methods
over 42 scenarios. The familywise error rate
control was respected for all methods in 39 out of 42 scenarios, for 2
scenarios it is slightly non-respected by all methods (7 or 8 runs with at
least a false selection out of 100). Considering Performance 1, we can see
that the single variable method performs slightly better for settings where the
correlation is not particularly high and the hierarchical methods perform
better for settings with high correlation. If one looks at Performance 2,
the disadvantage of the hierarchical methods in the ``normal 
correlation'' settings gets smaller (difference of 1.6 percent at
most and 0.2 percent on average), while their advantage in the ``high
correlation'' settings gets more substantial, with an average (over all
scenarios) improvement of 5 percent when considering canonical correlation
clustering and 7 percent when considering \texttt{hclust} clustering.
 
Taking a a more detailed and informative viewpoint in Section
\ref{sectionBigblocks}, the hierarchical method dominates the single
variable method in terms of minimal true detections 
(MTDs), while both method detects a similar number of singletons (the
hierarchical method being slightly preferable in this aspect, too). While 
both methods exhibit a good performance for the scenario generated with
$\rho=0$, the clear superiority of the hierarchical method becomes apparent for
increasing values of the correlations among the variables. The empirical
findings supporting this statement are supported with additional results
presented in the Supplemental Material. 

Applying the hierarchical methods to a real dataset about motif regression
\citep{Conlon03}, we obtained an indication that there might be other potential
motifs in a large cluster of size 165 which could play a significant role.
 
\section{Robustness of the method with respect to failure of variable screening}\label{sectionrobust}

The variable screening assumption (A2) seems far from necessary for
controlling the FWER as described in Theorem \ref{theo1}. Table
\ref{tableSNRhigh} provides empirical support for this fact. 

\subsection{A heuristic explanation}

The following argument yields some explanation why the
screening property is a too restrictive assumption. Let us assume that the
screening property fails because the beta-min condition (\ref{beta-min})
fails to hold. We then expect rather different selected sets
$\hat{S}^{(1)},\ldots,\hat{S}^{(B)}$, and the resulting p-values
$p_{adj}^{C,(1)},\ldots ,p_{adj}^{C,(B)}$ based on these selected sets are
likely to be rather different as well (since $\hat{S}^{(b)} \not \supseteq
S_0$ for most of the the $b$'s): many of them wouldn't exhibit a small value
and thus, when aggregating these p-values, the resulting aggregated p-value
is likely to be non-small. For example, when aggregating with the sample
median ($\gamma = 1/2$ in Section \ref{sec:aggregating}), more than 50\% of the p-values
would need to be small such that the aggregated value would be small as
well; and thus, the method only makes rejections if the single
p-values $p_{adj}^{C,(1)},\ldots ,p_{adj}^{C,(B)}$ are stable and a
substantial fraction of them are small (and hence, we expect conservative
behavior with respect to FWER control). We note that failure of (A2) due
to a different reason than failure of the beta-min condition
(\ref{beta-min}), such as ill-posed correlations among the variables, might
lead to stable p-values where a large fraction of them are spuriously
small: and in such a circumstance, the method might perform poorly with
respect to controlling the FWER. 

\subsection{A mathematical argument based on zonal assumptions}
We rigorously argue here that failure of the beta-min condition
(\ref{beta-min}) still leads to control of the FWER, assuming alternative
and weaker zonal assumptions \citep{bueman12}. 

We partition the active set $S_0$ into sets with corresponding large and
small regression coefficients, respectively: 
\begin{eqnarray*}
& &S_0 = S_{0,large}(a) \cup S_{0,small}(u),\\
& &S_{0,large}(a) = \{ j; |\beta_j^0|>a  \},\ \ S_{0,small}(u) = \{ j;
|\beta_j^0|\leq u  \}, 
\end{eqnarray*}
where $ 0 < u < a$. 

Consider the model $(\ref{mod1})$ with 
noise vector 
$\eps \sim \mathcal{N}(0,\sigma^2I)$. It can be rewritten as
\begin{eqnarray*}
\by = \bx^{\hat{S}} \beta_{\hat{S}}^0 + \bx^{\hat{S}^c} \beta_{\hat{S}^c}^0  + \eps,
\end{eqnarray*}
where $\hat{S}=\hat{S}(I_1) \subseteq \{1 \dots p \}$,
$|\hat{S}|\leq|I_2|$ and $\bx^{\hat{S}}$ the design sub-matrix of $\bx$
with columns corresponding to $\hat{S}$, and $I_1,\ I_2$ denote the two
sub-samples such that $I_1 \cup I_2 = \{1,\ldots
n\}$. Assume for the $|I_2| \times
|\hat{S}|$ design sub-matrix $\bx_{I_2}^{\hat{S}}$ of $\bx$ with rows
corresponding to $I_2$ and columns corresponding to $\hat{S}$:
\begin{eqnarray}\label{rank}
\rank\,( (\bx_{I_2}^{\hat{S}})^T \bx_{I_2}^{\hat{S}} )=|\hat{S}|.
\end{eqnarray}
Then define the following least squares estimates based on the sub-sample
$I_2$ and using only the variables from $\hat{S}$: 
\begin{eqnarray*}
& &\hat{\beta}_{I_2}^{\hat{S}} = \big((\bx_{I_2}^{\hat{S}})^T
\bx_{I_2}^{\hat{S}}\big)^{-1} (\bx_{I_2}^{\hat{S}})^T Y_{I_2},\\
& &P_{I_2}^{\hat{S}} = \bx_{I_2}^{\hat{S}} \big((\bx_{I_2}^{\hat{S}})^T
\bx_{I_2}^{\hat{S}}\big)^{-1} (\bx_{I_2}^{\hat{S}})^T,\ \ Q_{I_2}^{\hat{S}} =
I_{I_2} - P_{I_2}^{\hat{S}},\\  
& &\hat{Y}_{I_2}^{\hat{S}} = P_{I_2}^{\hat{S}} Y_{I_2} = \bx_{I_2}^{\hat{S}}
\hat{\beta}_{I_2}^{\hat{S}},\ \ \hat{\eps}_{I_2}^{\hat{S}} =
Q_{I_2}^{\hat{S}} Y_{I_2} = Y_{I_2} - 
\hat{Y}_{I_2}^{\hat{S}},\\
& &(\hat{\sigma}_{I_2}^{\hat{S}})^2 =
\frac{\|\hat{\eps}_{I_2}^{\hat{S}}\|_2^2}{|I_2|-|\hat{S}|}. 
\end{eqnarray*}
\begin{theo}\label{theo:robust}
Consider any selector $\hat{S}$ which is based on the sub-sample $I_1$ and
satisfies (\ref{rank}). Then, for a $q \times |\hat{S}|$-matrix $A$, 
$$\frac{(A\hat{\beta}_{I_2}^{\hat{S}}-A\beta_{\hat{S}}^0)^T
  \big(A\big(\bx_{I_2}^{\hat{S},T}
  \bx_{I_2}^{\hat{S}}\big)^{-1}A^T\big)^{-1}(A\hat{\beta}_{I_2}^{\hat{S}}-A\beta_{\hat{S}}^0
)}{q(\hat{\sigma}_{I_2}^{\hat{S}})^2} \sim
F_{q,|I_2|-|\hat{S}|}(\lambda_{\text{noncentral}})$$
is noncentral $F$-distributed with noncentrality parameter
\begin{eqnarray*}
& &\lambda_{\mbox{noncentral}}=\sum_{i=1}^q (\mbox{BIAS})_i^2,\\
& &\mbox{BIAS} = \frac{1}{\sigma} \big(A \big((\bx_{I_2}^{\hat{S}})^T
\bx_{I_2}^{\hat{S}}\big)^{-1} A^T \big)^{-1/2} A \big((\bx_{I_2}^{\hat{S}})^T
\bx_{I_2}^{\hat{S}}\big)^{-1}
(\bx_{I_2}^{\hat{S}})^T\bx_{I_2}^{\hat{S}^c}\beta_{\hat{S}^c}^0.
\end{eqnarray*} 
\end{theo}
A proof is given in the Supplemental Material.
Theorem \ref{theo:robust} gives the distribution of the partial F-test
statistic in the general case where a failure of screening is possible. The
noncentrality parameter $\lambda_{\text{noncentral}}$, however, is unknown
in practice. Clearly, if $\hat{S} \supseteq S_0$, then $\beta_{\hat{S}^c}^0
= 0$ and the noncentrality parameter
$\lambda_{\text{noncentral}}=0$. Thus, if $\hat{S}$ is approximately
correct for screening $S_0$, then $\lambda_{\text{noncentral}} \approx 0$. 

In the following example we show that considering the Lasso as screening
procedure and assuming zonal assumptions on the active variables, 
Theorem \ref{theo:robust} implies asymptotically valid p-values when taking
a partial F-test with central F-distribution (i.e, the noncentrality
parameter is asymptotically negligible).

\subsection{The Lasso as selector $\hat{S}$ and zonal assumptions for
  $\beta^0$} 

For the Lasso, assuming that the
compatibility condition holds with compatibility constant 
$\phi^2_0 >0$ \citep[cf.(6.4)]{pbvdg11}, with probability tending to one:
$$\| \hat{\beta} - \beta^0 \|_\infty \le \| \hat{\beta} - \beta^0 \|_1 \le 
a(n,p,s_0,\bx,\sigma) := C\sigma s_0 \sqrt{\log(p)/n}/\phi^2_0$$
for some $C=C(\lambda)>0$ when choosing the regularization parameter
$\lambda \asymp \sigma \sqrt{\log(p)/n}$ \citep[Th6.1]{pbvdg11}. Hence on
an event with high probability, we 
have for this $a=a(n,p,s_0,\bx,\sigma)$, 
$$\hat{S} \supseteq S_{0,large}(a)$$ 
\citep{bueman12} and using the partitioning of $S_0$ it follows that 
$${\|\beta_{\hat{S}^c}^0\|}_\infty \leq u
\mbox{ and } {\|\beta_{\hat{S}^c}^0\|}_0 \leq s_{0,small}(u).$$
Assuming constants $C_1, C_2$ and $C_3$ such that
\begin{eqnarray}\label{add1}
& &\max_{j=1,\dots,p} (\bx_{I_2}^T\bx_{I_2})_{jj} \leq C_1|I_2|\nonumber\\
& &\max_{j,k \in \hat{S}} |\big(\bx_{I_2}^{\hat{S},T}
\bx_{I_2}^{\hat{S}}\big)^{-1}|_{jk} \leq C_2|I_2|^{-1}\nonumber\\
& &\max_{j,k=1,\dots,q} |\big(A \big(\bx_{I_2}^{\hat{S},T}
\bx_{I_2}^{\hat{S}}\big)^{-1} A^T \big)^{-1/2}|_{jk} \leq C_3|I_2|^{1/2}\nonumber\\
& &\mbox{for each } q \times |\hat{S}|\mbox{-matrix } A \mbox{ with } q<|\hat{S}|,~A_{jj}\in \{0,1\} \mbox{ and } A_{jk}=0 \mbox{ for } j \neq k.\nonumber\\ 
\end{eqnarray} 
it follows for the noncentrality parameter
\begin{eqnarray*}
\lambda_{noncentral} &\leq& q \max_{i=1,\dots,q} (\mbox{BIAS})_i^2\\
&\leq& q \big( \frac{1}{\sigma} C_3 |I_2|^{1/2} |\hat{S}| C_2|I_2|^{-1} |\hat{S}|
C_1|I_2| s_{0,small}(u) u \big)^2\\
&\leq& \Big( \frac{C_1C_2C_3}{\sigma} |\hat{S}|^{5/2} |I_2|^{1/2}
s_{0,small}(u) u \Big)^2
\end{eqnarray*}
Now, assuming a more restrictive sparse eigenvalue condition on the design
$\bx$ we have $|\hat{S}| \leq C_4s_0$ for some constant $0<C_4<\infty$
\citep{zhang2008sparsity,geer11} and hence for some constant 
$D=D(C_1,C_2,C_3,C_4)$  
\begin{eqnarray*}
\lambda_{noncentral} &\leq& \big( \frac{D}{\sigma}
 s_0^{5/2}s_{0,small}(u) \sqrt{n}u \big)^2, 
\end{eqnarray*}
i.e. the noncentrality parameter is negligible for $u$ being at most of
small order $o(n^{-1/2})$. Note that the inequality above is implicit in
the value $u$ since it involves $s_{0,small}(u)$: of course, we can give
the upper bound 
\begin{eqnarray*}
\lambda_{noncentral} &\leq& \big( \frac{D}{\sigma}
 s_0^{7/2} \sqrt{n}u \big)^2, 
\end{eqnarray*}
implying that $u = o(s_0^{-7/2}n^{-1/2})$ suffices to obtain asymptotic
negligibility of the noncentrality parameter. 

We conclude as follows. Assume that (\ref{rank}), (\ref{add1}) hold and
that the design matrix satisfies a sparse eigenvalue condition with sparse
eigenvalue bounded away from zero. Furthermore, replace the screening
property in (A2) by zonal assumptions for the regression coefficients:
\begin{eqnarray*}
& &S_0 = S_{0,large}(a) \cup S_{0,small}(u), \mbox{ with}\\
& &a=C\sigma s_0 \sqrt{\log(p)/n}/\phi^2_0 \mbox{ for } C>0 \mbox{
  sufficiently large,}\\
& &u=\tilde{C}\sigma s_0^{-5/2}s^{-1}_{0,small}(u)n^{-1/2} \mbox{ for }
\tilde{C}>0 \mbox{ sufficiently small.}
\end{eqnarray*}
Then, when using the Lasso as selector $\hat{S}$, our hierarchical p-value
method provides asymptotic strong error control of the familywise error
rate. 

\section{Conclusions}
We propose a method for testing whether (mainly) groups of correlated
variables are significant for explaining a response in a high-dimensional
linear model. In presence of highly
correlated variables (or nearly collinear smaller groups of variables), as
is very common in high-dimensional data, it seems indispensable to adopt
such a kind of an approach going beyond multiple testing of individual
regression coefficients. The groups of variables  
are ordered within a given hierarchy, for example a cluster tree, which
allows for powerful multiple testing adjustment. It automatically
determines a good 
resolution level distinguishing between small and large groups of
variables: the former are significant if the signal of one or few
individual variables in such a small group is strong and/or the variables
are not too highly correlated; and a large group can be significant even if 
the signals of (many) individual variables in the group are weak and the 
variables exhibit high correlation among themselves. The minimal true
detections (MTDs) measure the power to detect significant smallest groups
of variables, and our method performs well in terms of MTDs and
substantially better than the analogue of a single variable method. 

Our procedure is based
on repeated sample splitting which was
empirically found to be ``robust'' and reliable for controlling type I
errors. We present some theory proving strong control of the familywise
error rate, and our assumptions allow for scenarios beyond the beta-min
condition saying that all non-zero regression 
coefficients should be sufficiently large. We also provide empirical
results for simulated and real data which complement the theoretical
analysis. 


\bigskip\noindent
\textbf{Acknowledgments:} We thank Nicolai Meinshausen and Patric
M\"uller for interesting comments and discussions. Furthermore, we thank
some anonymous reviewers for constructive and insightful comments.

\section{Supplemental Materials}
\begin{description}
\item[{\parbox[t]{1\linewidth}{Supplemental Material for ``Hierarchical
    Testing in the\\ High-Dimensional Setting with Correlated
    Variables'':}}]\mbox{}\vspace{5mm}\\
An alternative
bottom-up hierarchical adjustment. Variability of Performance 1 and
Performance 2 in the simulations study. Variability of MTDs in Section
\ref{sectionBigblocks}. Extension of the considerations of Section
\ref{sectionBigblocks} for low SNR. Proofs. (pdf file)
\end{description}

\bibliographystyle{apalike}
\bibliography{referencesHierarchical}

\begin{thebibliography}{}

\bibitem[Benjamini and Hochberg, 1995]{benjamini1995controlling}
Benjamini, Y. and Hochberg, Y. (1995).
\newblock Controlling the false discovery rate: a practical and powerful
  approach to multiple testing.
\newblock {\em Journal of the Royal Statistical Society. Series B
  (Methodological)}, pages 289--300.

\bibitem[B{\"u}hlmann, 2013]{bue12}
B{\"u}hlmann, P. (2013).
\newblock Statistical significance in high-dimensional linear models.
\newblock {\em Bernoulli}, 19:1212--1242.

\bibitem[B\"uhlmann et~al., 2014]{bumeka13}
B\"uhlmann, P., Kalisch, M., and Meier, L. (2014).
\newblock High-dimensional statistics with a view toward applications in
  biology.
\newblock {\em Annual Review of Statistics and Its Application}, 1:255--278.

\bibitem[B{\"u}hlmann and Mandozzi, 2014]{bueman12}
B{\"u}hlmann, P. and Mandozzi, J. (2014).
\newblock High-dimensional variable screening and bias in subsequent inference,
  with an empirical comparison.
\newblock {\em Computational Statistics}, 29:407--430.

\bibitem[{B{\"u}hlmann} et~al., 2013]{buru12}
{B{\"u}hlmann}, P., {R{\"u}timann}, P., {van de Geer}, S., and {Zhang}, C.-H.
  (2013).
\newblock Correlated variables in regression: clustering and sparse estimation
  (with discussion).
\newblock {\em Journal of Statistical Planning and Inference}, 143:1835--1871.

\bibitem[B\"uhlmann and van~de Geer, 2011]{pbvdg11}
B\"uhlmann, P. and van~de Geer, S. (2011).
\newblock {\em Statistics for High-Dimensional Data: Methods, Theory and
  Applications}.
\newblock Springer Verlag, New York, NY.

\bibitem[Chatterjee and Lahiri, 2013]{chala13}
Chatterjee, A. and Lahiri, S.~N. (2013).
\newblock Rates of convergence of the adaptive {L}asso estimators to the oracle
  distribution and higher order refinements by the bootstrap.
\newblock {\em Annals of Statistics}, 41:1232--1259.

\bibitem[Conlon et~al., 2003]{Conlon03}
Conlon, E.~M., Liu, X.~S., Lieb, J.~D., and Liu, J.~S. (2003).
\newblock Integrating regulatory motif discovery and genome-wide expression
  analysis.
\newblock {\em Proceedings of the National Academy of Sciences},
  100:3339--3344.

\bibitem[Dezeure et~al., 2014]{DeBuMeMe14}
Dezeure, R., B\"uhlmann, P., Meier, L., and Meinshausen, N. (2014).
\newblock High-dimensional {I}nference: Confidence intervals, p-values and
  {R}-software hdi.
\newblock arXiv:1408.4026v1.

\bibitem[Goeman et~al., 2006]{goeman2006testing}
Goeman, J.~J., Van De~Geer, S.~A., and Van~Houwelingen, H.~C. (2006).
\newblock Testing against a high dimensional alternative.
\newblock {\em Journal of the Royal Statistical Society, Series B},
  68:477--493.

\bibitem[Javanmard and Montanari, 2014a]{javmo13}
Javanmard, A. and Montanari, A. (2014a).
\newblock Confidence intervals and hypothesis testing for high-dimensional
  regression.
\newblock arXiv:1306.3171v2, To appear in Journal of Machine Learning Research.

\bibitem[Javanmard and Montanari, 2014b]{jamo13}
Javanmard, A. and Montanari, A. (2014b).
\newblock Hypothesis testing in high-dimensional regression under the
  {G}aussian random design model: asymptotic theory.
\newblock arXiv:1301.4240v3, To appear in IEEE Transaction on Information
  Theory.

\bibitem[Liu and Yu, 2013]{liuyu13}
Liu, H. and Yu, B. (2013).
\newblock Asymptotic properties of {L}asso+m{LS} and {L}asso+{R}idge in sparse
  high-dimensional linear regression.
\newblock {\em Electronic Journal of Statistics}, 7:3124--3169.

\bibitem[Meinshausen, 2008]{Meins08}
Meinshausen, N. (2008).
\newblock Hierarchical testing of variable importance.
\newblock {\em Biometrika}, 95:265--278.

\bibitem[Meinshausen, 2013]{meins13}
Meinshausen, N. (2013).
\newblock Assumption-free confidence intervals for groups of variables in
  sparse high-dimensional regression.
\newblock arXiv:1309.3489v1.

\bibitem[Meinshausen et~al., 2009]{memepb09}
Meinshausen, N., Meier, L., and B{\"u}hlmann, P. (2009).
\newblock P-values for high-dimensional regression.
\newblock {\em Journal of the American Statistical Association},
  104:1671--1681.

\bibitem[Minnier et~al., 2011]{mititi11}
Minnier, J., Tian, L., and Cai, T. (2011).
\newblock A perturbation method for inference on regularized regression
  estimates.
\newblock {\em Journal of the American Statistical Association},
  106:1371--1382.

\bibitem[Shaffer, 1986]{Shaf86}
Shaffer, J.~P. (1986).
\newblock Modified sequentially rejective multiple test procedures.
\newblock {\em Journal of the American Statistical Association}, 81:826--831.

\bibitem[Tibshirani, 1996]{T96}
Tibshirani, R. (1996).
\newblock Regression shrinkage and selection via the {L}asso.
\newblock {\em Journal of the Royal Statistical Society, Series B},
  58:267--288.

\bibitem[van~de Geer et~al., 2014]{geeretal13}
van~de Geer, S., B{\"u}hlmann, P., Ritov, Y., and Dezeure, R. (2014).
\newblock On asymptotically optimal confidence regions and tests for
  high-dimensional models.
\newblock {\em The Annals of Statistics}, 42:1166--1202.

\bibitem[{van de Geer} et~al., 2011]{geer11}
{van de Geer}, S., B{\"u}hlmann, P., and Zhou, S. (2011).
\newblock The adaptive and the thresholded {L}asso for potentially misspecified
  models (and a lower bound for the {L}asso).
\newblock {\em Electronic Journal of Statistics}, 5:688--749.

\bibitem[{van 't Veer} et~al., 2002]{vV2002}
{van 't Veer}, L.~J., Dai, H., van~de Vijver, M.~J., He, Y.~D., Hart, A. A.~M.,
  Mao, M., Peterse, H.~L., van~der Kooy, K., Marton, M.~J., Witteveen, A.~T.,
  Schreiber, G.~J., Kerkhoven, R.~M., Roberts, C., Linsley, P.~S., Bernards,
  R., and Friend, S.~H. (2002).
\newblock Gene expression profiling predicts clinical outcome of breast cancer.
\newblock {\em Nature}, 415:530--536.

\bibitem[Wasserman and Roeder, 2009]{WR08}
Wasserman, L. and Roeder, K. (2009).
\newblock {High dimensional variable selection}.
\newblock {\em Annals of Statistics}, 37:2178--2201.

\bibitem[Zhang and Huang, 2008]{zhang2008sparsity}
Zhang, C.-H. and Huang, J. (2008).
\newblock {The sparsity and bias of the Lasso selection in high-dimensional
  linear regression}.
\newblock {\em Annals of Statistics}, 36:1567--1594.

\bibitem[Zhang and Zhang, 2014]{zhazha13}
Zhang, C.-H. and Zhang, S.~S. (2014).
\newblock Confidence intervals for low dimensional parameters in high
  dimensional linear models.
\newblock {\em Journal of the Royal Statistical Society, Series B},
  76:217--242.

\end{thebibliography}

\section*{Supplemental material to Section \ref{sectiondescription}}

%
%

\subsection*{An alternative bottom-up hierarchical adjustment}
The procedure described in Section \ref{sectiondescription} is based on a 
top-down hierarchical adjustment of the p-values 
$P_h^C = \max_{D \in \mathcal{T}:C \subseteq D} P^C$. 
Another possibility is the following bottom-up approach.

We begin with clustering as in Section \ref{subsec.clustering} and screening
as in Section \ref{sec:Screening}. Then we take the p-values
$p^{C,(b)}$ as in (\ref{eq:pv}) and define
$$\overline{p}^{C,(b)}_h = \min\{\, 2|\hat{S}^{(b)}| \min_{D \in
  \mathcal{T}:D \subseteq
  C} p^{C,(b)}\,,\,1\}.$$
Finally we define for $\gamma \in (0,1)$ the aggregated p-values
$$\overline{Q}_h^C(\gamma) = \min \big\{ \,1~,~q_\gamma \big(
\big\{\overline{p}_h^{C, (b)} / \gamma;\, b=1,\dots,B \big\} \big) \big\}$$
and eliminate $\gamma$ taking
$$\overline{P}_h^C = \min \big\{ \,1~,~(1-\log\gamma_{\min}) \inf_{\gamma \in
  (\gamma_{\min},1)} \overline{Q}_h^C(\gamma) \big\}.$$
The price one has to pay for minimizing among p-values of children clusters
instead of maximizing among p-values of parents clusters is a factor $|C
\cap \hat{S}^{(b)}|$ in the multiplicity adjustment.

Although none of the two methods theoretically dominates the other,
simulations with some scenarios as in Section \ref{sectionempirical} have
shown that the top-down method exhibits substantially higher power than the
bottom-up method. Hence we put our focus on the top-down method.
%

\section*{Supplemental material to Section \ref{sectionempirical}}

\subsection*{Variability of Performance 1 and Performance 2 in the
  simulation study}

To give some idea about the variability among the different simulation
runs, we show in Figures \ref{figurePerformance1} and
\ref{figurePerformance2} the Performance 1 and Performance 2 measures,
respectively for all 100 runs of some of the scenarios.

\begin{figure*}[!htp]
\centerline{
\includegraphics[width=0.8\textwidth, angle=0]{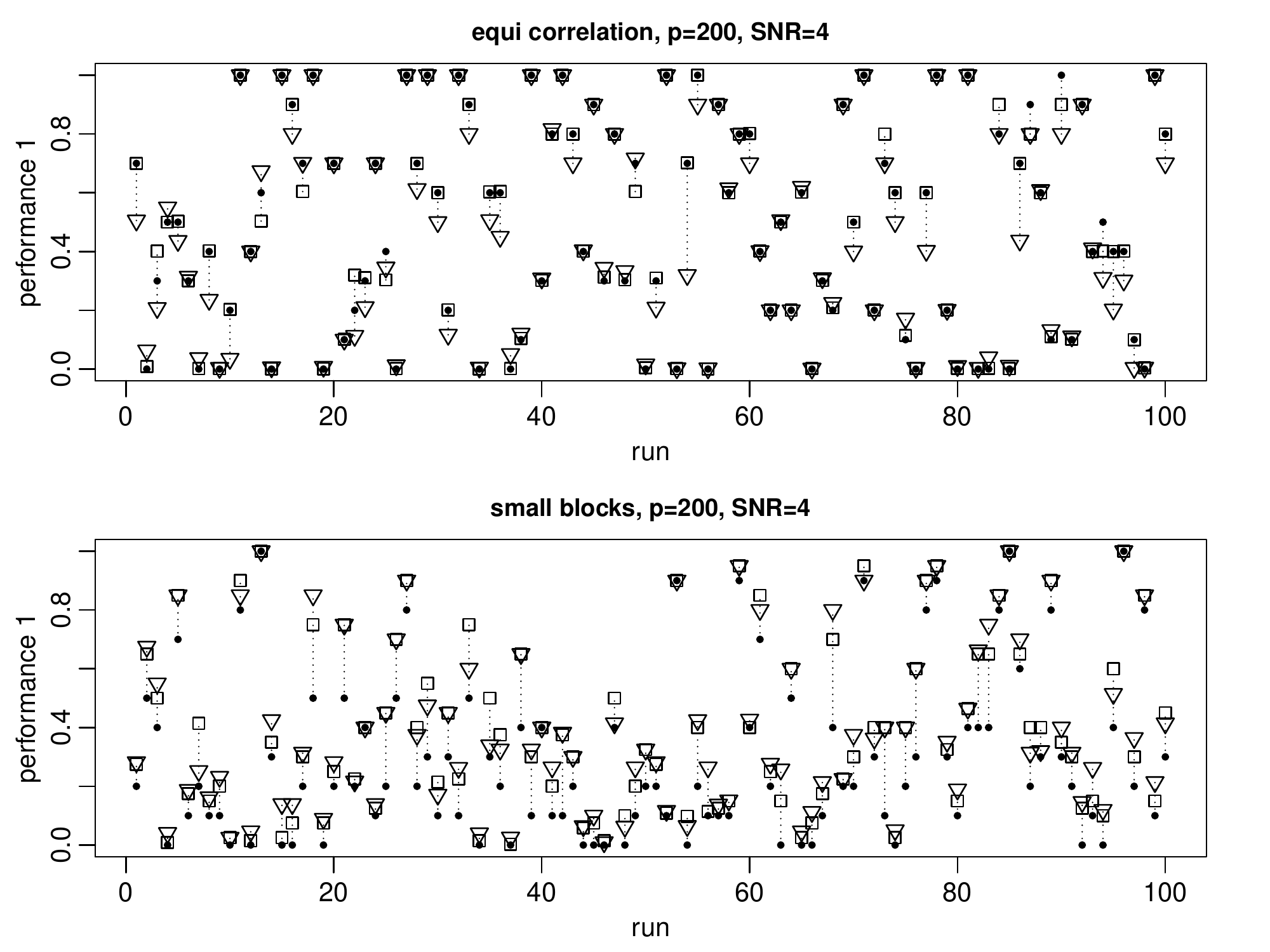}}
\caption{The Performance 1 measure for all 100 runs for 2 different
  scenarios described in the header of the plots. Single variable method
  (filled small circle), the hierarchical method with canonical correlation 
  clustering (empty square) and \texttt{hclust} clustering (triangle).}
\label{figurePerformance1}
\end{figure*}

In Figure \ref{figurePerformance1} we consider Performance 1 for two
synthetic scenarios, one where the single variable method is favored and
another where the hierarchical method is better. In Figure
\ref{figurePerformance2} we adopt the same approach for Performance 2
considering two scenarios based on semi-real datasets. 

\begin{figure*}[!htp]
\centerline{
\includegraphics[width=0.8\textwidth, angle=0]{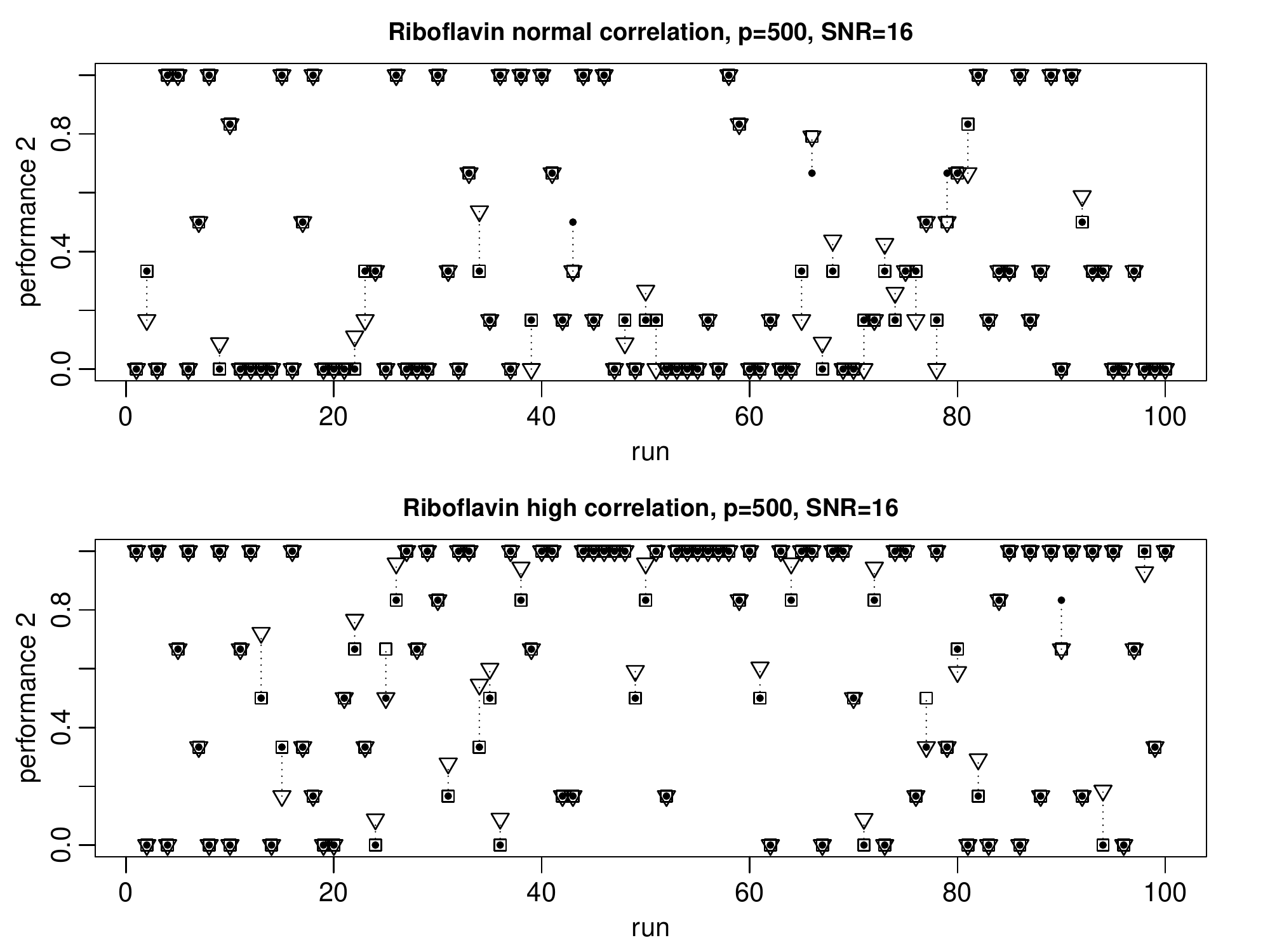}}
\caption{The Performance 2 measure for all 100 runs for 2 different
  scenarios described in the header of the plots. Single variable method
  (filled small circle), the hierarchical method with canonical correlation 
  clustering (empty square) and \texttt{hclus} clustering (triangle).}
\label{figurePerformance2}
\end{figure*}

\subsection*{Variability of MTDs in Section \ref{sectionBigblocks}}

We show in Figures \ref{figuresmallblockshighSNR2} and
  \ref{figurebigblockshighSNR2} the number of MTDs for all 
simulation runs of the ``small blocks''-design with $\mbox{SNR}=8$ and
$\rho=0.7$ and $\rho=0.95$, respectively, and for the ``large
blocks''-design with $\mbox{SNR}=8$ and $\rho=0.4$ and $\rho=0.9$,
respectively. For each of the 100 simulation runs and cardinalities from 1
to 20,
the number of MTDs for the hierarchical method with \texttt{hclus}
clustering is depicted in black while the number of MTDs for the single 
variable method is depicted in gray, for graphical convenience at the
bottom of the y-axis (since the cardinality of the MTDs of the single
variable method is always equal to 1). 

\begin{sidewaysfigure}[!htp]
\centerline{\includegraphics[width=1\textwidth, angle=0]{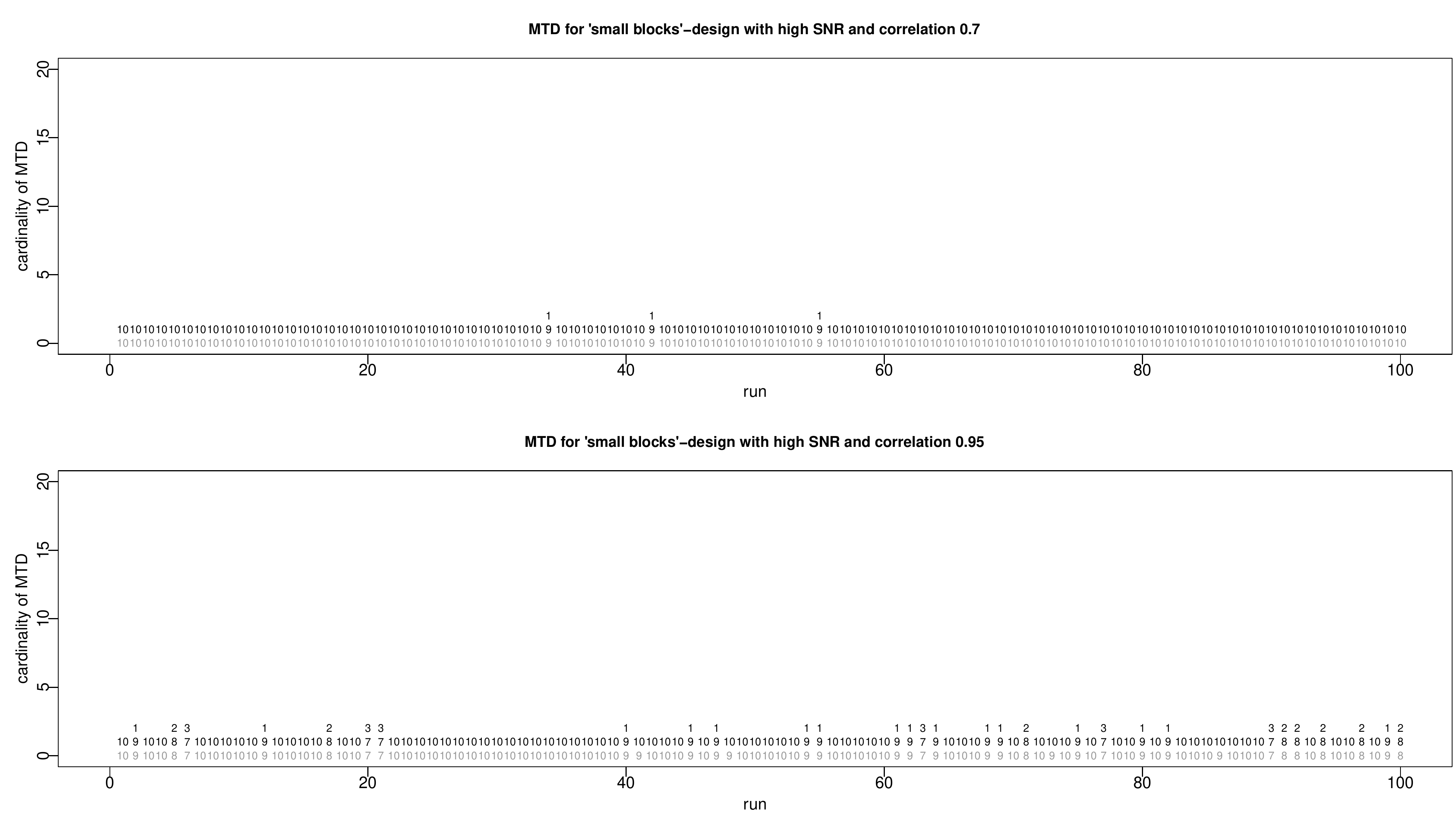}}
\caption{Number of MTDs for ``small blocks''-design with high SNR (SNR=8)
  and $\rho=0.7$ resp.
  $\rho=0.95$. For each of the 100 simulation runs (x-axis) and every
  cardinality 
  (y-axis), the number of MTDs for the hierarchical method with
  \texttt{hclus} clustering (in black) and for the single variable method
  (in gray, 
  for graphical convenience at the bottom of the y-axis).}
\label{figuresmallblockshighSNR2}
\end{sidewaysfigure}

\begin{sidewaysfigure}[!htp]
\centerline{\includegraphics[width=1\textwidth, angle=0]{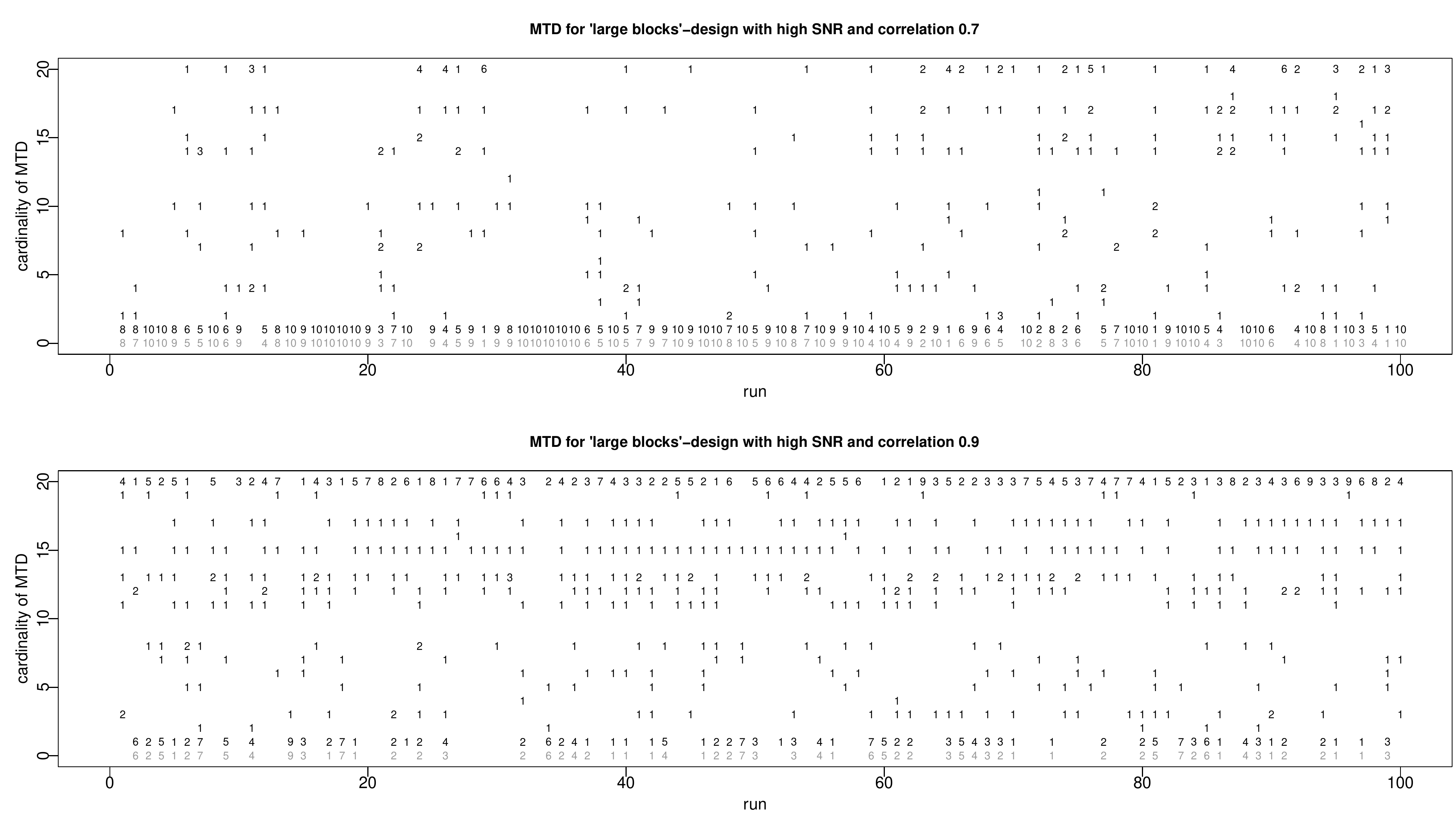}}
\caption{Number of MTDs for ``large blocks''-design with high SNR (SNR=8) and
  $\rho=0.7$ resp.
  $\rho=0.9$. For each of the 100 simulation runs (x-axis) and every cardinality
  (y-axis), the number of MTDs for the hierarchical method with
  \texttt{hclus} clustering (in black) and for the single variable method
  (in gray, for graphical convenience at the bottom of the y-axis).}
\label{figurebigblockshighSNR2}
\end{sidewaysfigure}

\subsection*{Extension of the considerations of Section
  \ref{sectionBigblocks} for low \mbox{SNR}}

We present here the same detailed analysis as in Section
\ref{sectionBigblocks} for the signal to noise ratio $\mbox{SNR}=4$. The
empirical results presented below show that the power of all considered
methods is significantly affected by the change of 
$\mbox{SNR}$ (e.g. for the ``large blocks''-design with $\rho \geq 0.7$
detecting at least one singleton is difficult when $\mbox{SNR}=4$), but
they also confirm the superiority of the hierarchical in comparison to the
single variable methods reported in the main paper in Section
\ref{sectionBigblocks}. 

Table \ref{tableSNRsmall} reports some average results over 100 simulation
runs. As for the case in the main
paper with high SNR, the number of singleton detections are again similar
for all methods. The large number of MTDs with cardinality 2 in the ``small
blocks''-design emphasizes the powerful advantage of automatically going to
the finer possible resolution with the hierarchical method. 

\begin{sidewaystable}[!htp]
\begin{tabular}{|c|c||c|c|c||c|c|c||c|c|c||c|c||c|c||c|c|}
\hline
 & & \multicolumn{3}{c||}{FWER} &
\multicolumn{3}{c||}{\# MTD} &
\multicolumn{9}{c|}{\# MTD for given cardinality}\\
\cline{9-17}
$\rho$ & $\delta$ & \multicolumn{3}{c||}{} & \multicolumn{3}{c||}{} & \multicolumn{3}{c||}{$|\cdot|=1$} &
 \multicolumn{2}{c||}{$|\cdot|=2$} &
 \multicolumn{2}{c||}{$3 \leq |\cdot| \leq 10$} & 
 \multicolumn{2}{c|}{$11 \leq |\cdot| \leq 20$}
\\
\cline{3-17}
 & & S & C & H & S & C & H & S & C & H & C & H & C & H & C &
 H\\
\hline\hline
\multicolumn{17}{|c|}{``small blocks''-design with low SNR}\\
0 & 0.28 & 0 & 0 & 0 & 8.78 & 8.85 & 8.79 & 8.78 & 8.72 & 8.49 & 0 & 0.03 & 0 & 0.09
& 0 & 0.07\\
\hline
0.4 & 0.18 & 0 & 0 & 0 & 8.74 & 9.11 & 8.82 & 8.74 & 8.83 & 8.47 & 0.26 &
0.05 & 0.02 & 0.08 & 0 & 0.01\\
\hline
0.7 & 0.45 & 0 & 0 & 0 & 4.80 & 6.89 & 7.02 & 4.80 & 5.26 & 4.86 & 1.41 & 1.21 &
0 & 0.56 & 0.19 & 0\\
\hline
0.8 & 0.49 & 0 & 0 & 0 & 4.74 & 7.13 & 7.41 & 4.74 & 4.95 & 4.78 & 1.99 &
2.00 & 0.02 & 0.55 & 0.16 & 0.04\\
\hline
0.85 & 0.38 & 0.03 & 0.03 & 0.03 & 6.03 & 7.84 & 8.00 & 6.03 & 6.39 & 6.10
& 1.41 & 1.70 & 0.04 & 0.20 & 0 & 0\\
\hline
0.9 & 0.53 & 0.05 & 0.05 & 0.07 & 4.31 & 7.07 & 7.47 & 4.31 & 4.63 & 4.67 &
2.22 & 2.23 & 0.06 & 0.56 & 0.16 & 0\\
\hline
0.95 & 0.98 & 0.33 & 0.36 & 0.42 & 1.29 & 3.83 & 4.82 & 1.29 & 1.48 & 1.28
& 1.95 & 2.02 & 0 & 1.08 & 0.26 & 0\\ 
\hline
0.99 & 0.94 & 0.46 & 0.60 & 0.54 & 3.96 & 6.47 & 6.66 & 3.96 & 4.06 & 3.63 &
2.26 & 2.28 & 0.15 & 0.39 & 0 & 0.13\\
\hline
\hline
\multicolumn{17}{|c|}{``large blocks''-design with low SNR}\\
0 & 0.24 & 0 & 0 & 0 & 8.18 & 8.20 & 8.23 & 8.18 & 8.03 & 7.86 & 0 & 0.18 &
0 & 0.25 & 0 & 0.09\\
\hline
0.4 & 0.35 & 0 & 0 & 0 & 4.51 & 5.06 & 6.90 & 4.51 & 4.53 & 4.26 & 0.01 &
0.14 & 0.06 & 0.46 & 0.24 & 1.31\\
\hline
0.7 & 0.76 & 0 & 0 & 0 & 0.28 & 2.84 & 4.69 & 0.28 & 0.30 & 0.27 & 0 &
0.06 & 0.11 & 0.26 & 1.88 & 2.36\\
\hline
0.8 & 0.72 & 0 & 0 & 0 & 0.58 & 4.93 & 6.48 & 0.58 & 0.63 & 0.60 & 0 &
0.07 & 0.35 & 0.66 & 3.67 & 3.83\\
\hline
0.85 & 0.82 & 0 & 0 & 0.01 & 0.48 & 5.53 & 6.42 & 0.48 & 0.48 & 0.46 & 0 & 0.04 &
0.44 & 0.54 & 4.40 & 4.30\\
\hline
0.9 & 0.99 & 0 & 0 & 0 & 0.08 & 3.6 & 4.61 & 0.08 & 0.10 & 0.08 & 0 &
0 & 0.16 & 0.17 & 2.77 & 2.77\\
\hline
0.95 & 1.00 & 0 & 0.23 & 0.15 & 0.10 & 6.98 & 7.25 & 0.10 & 0.10 & 0.10 & 0.01 &
0.03 & 0.27 & 0.87 & 6.40 & 5.86\\ 
\hline
0.99 & 1.00 & 0.93 & 1.00 & 1.00 & 1.60 & 5.00 & 5.10 & 1.60 & 1.60 & 1.60 &
0 & 0.24 & 1.49 & 2.21 & 1.91 & 1.05\\
\hline 
\end{tabular}
\caption{Results of the simulation with the ``large blocks''- and ``small
  blocks'' -design with low SNR (SNR=4) for different correlations. 
$\rho$ is the correlation in the design, $\delta$
the relative frequency of screenings with $\hat{S} \not\supset S_0$, MTD
denotes ``minimal true detections'', ``$2 \leq |\cdot| \leq 5$'' indicates
that MTD of cardinality between 2 and 5 are considered, S, C and H
represent the ``single variable'' resp. ''canonical correlation
clustering`` and ``hierarchical with \texttt{hclus} clustering'' method.}
\label{tableSNRsmall}
\end{sidewaystable}

To better illustrate what happens in a typical simulation run, we show in Figure
\ref{figurebigblockslowSNR1} the
dendrograms for a representative
simulation run of the ``large blocks''-design with $\rho=0.85$ (here with
$\mbox{SNR}=4$), for the single variable method and the hierarchical method
with \texttt{hclus} clustering. The active variables are labeled in black
and the truly detected non-zero variables
along the hierarchy are depicted in black. While the single 
variable method ``only'' detects one singleton, the hierarchical method
detects the same singleton and achieves 8 more MTDs.  

\begin{figure*}[!htp]
\centerline{\includegraphics[width=1.1\textwidth, angle=0]{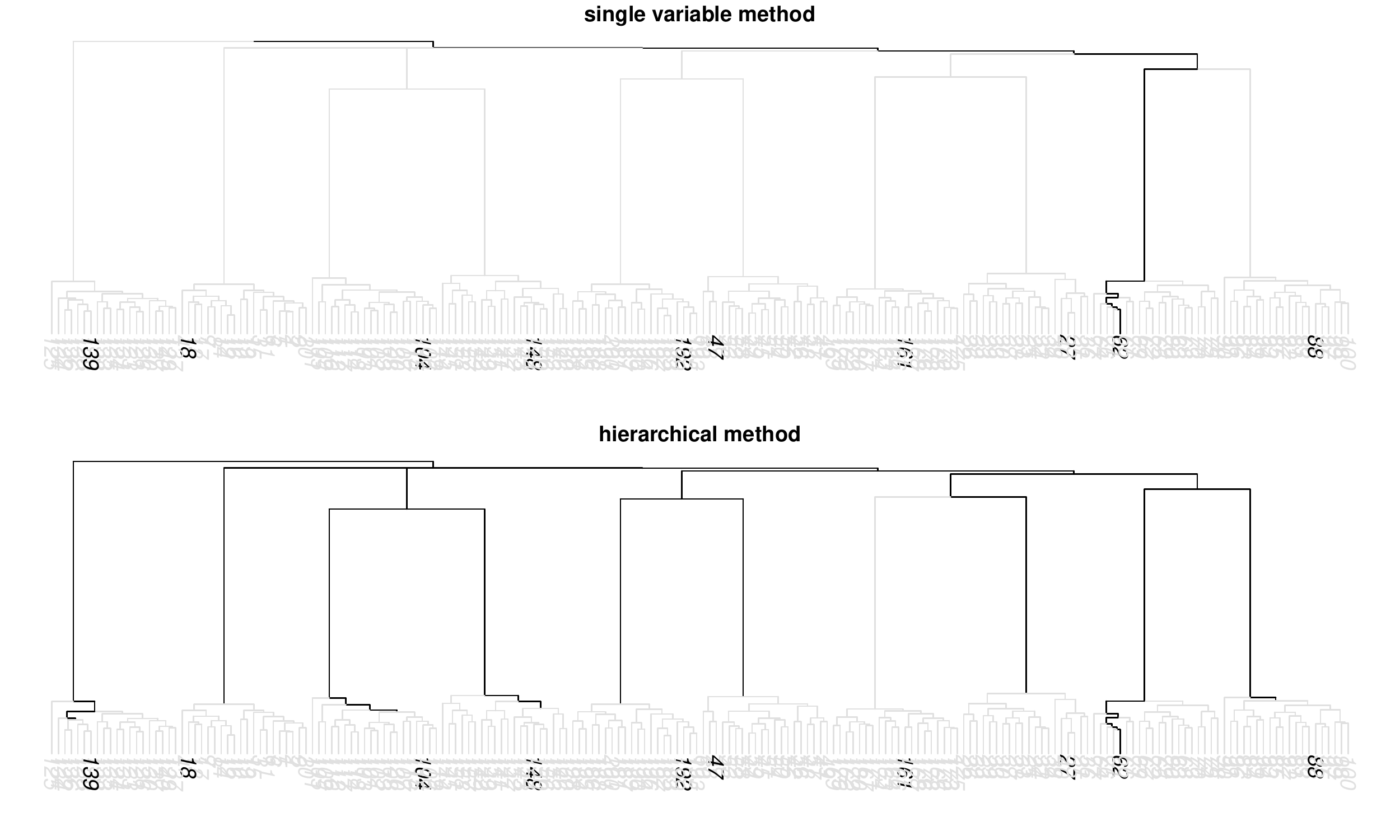}}
\vspace{-0.5cm}
\caption{Dendrograms for a representative run of the ``large
  blocks''-design with low SNR (SNR=4) and $\rho=0.85$. The active
  variables are labeled in black and the truly detected
  non-zero variables along the hierarchy are depicted in black.}
\label{figurebigblockslowSNR1}
\end{figure*}

Figure \ref{figuresmallblockslowSNR1} is the analogous of Figure
\ref{figurebigblockslowSNR1} in the main paper for a typical run of the
``small blocks''-design with 
$\rho=0.8$. It shows that the hierarchical method improves the results of
the single variable method (which detects 5 singletons) providing 3 more MTDs of
cardinality 2.

\begin{figure*}[!htp]
\centerline{\includegraphics[width=1.1\textwidth, angle=0]{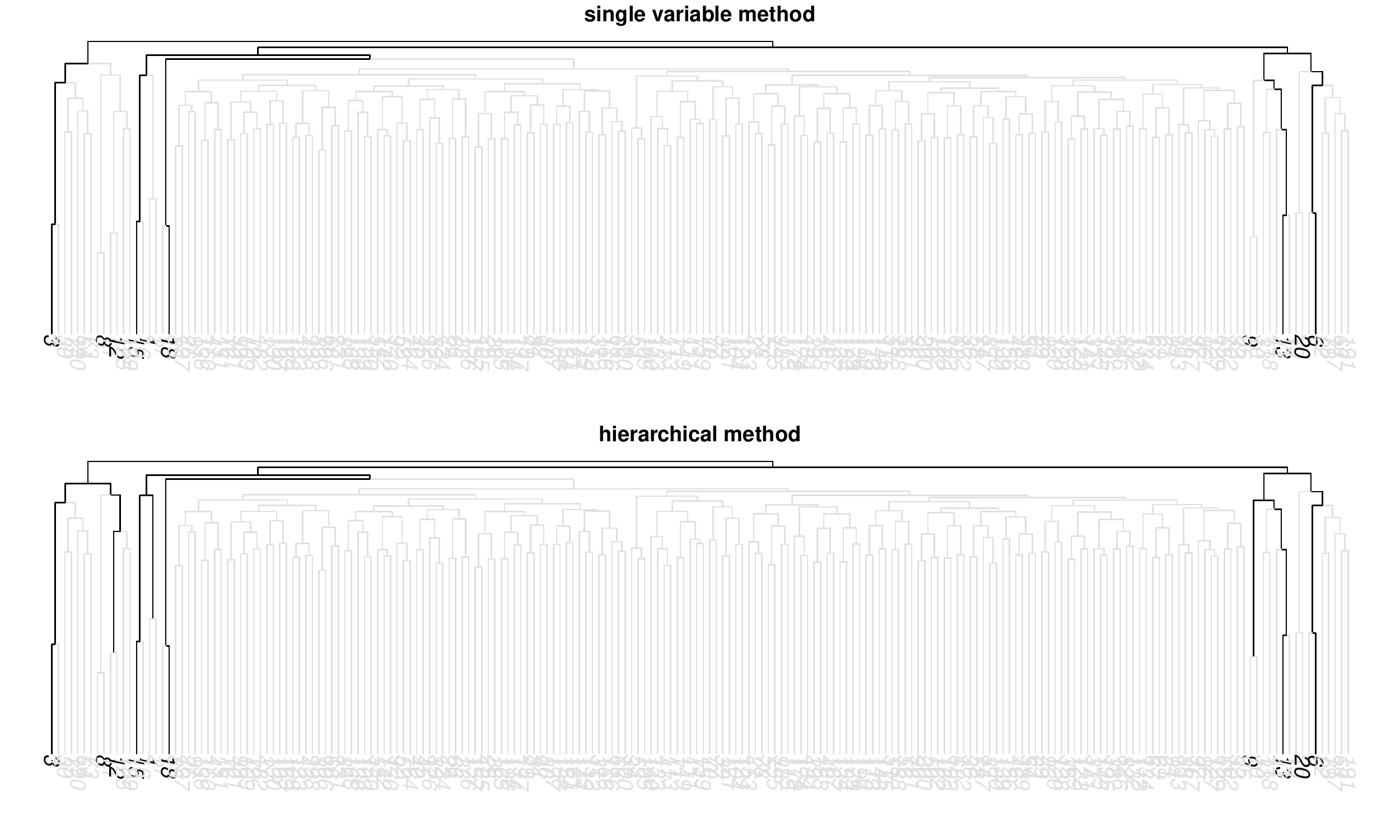}}
\vspace{-0.5cm}
\caption{Dendrograms for a representative run of the ``small
  blocks''-design with low SNR (SNR=4) and $\rho=0.80$. The active
  variables are labeled in black and the truly detected 
non-zero variables along the hierarchy are depicted in black.}
\label{figuresmallblockslowSNR1}
\end{figure*}

In Figures \ref{figuresmallblockslowSNR2} and \ref{figurebigblockslowSNR2} and we show the number of MTDs for all
100 simulation runs of the ``small blocks''-design with $\mbox{SNR}=4$ and
$\rho=0.7$ and $\rho=0.9$, respectively. and for the ``large
blocks''-design with $\mbox{SNR}=4$ and $\rho=0.4$ and $\rho=0.9$,
respectively. For each simulation run and cardinalities from 1 to 20,
the number of MTDs for the hierarchical method with \texttt{hclus}
clustering is depicted in black while the number of MTDs for the single 
variable method is depicted in gray, for graphical convenience at the
bottom of the y-axis (since the cardinality of MTDs of the single variable method
is always equal to 1). 

\begin{sidewaysfigure}[!htp]
\centerline{\includegraphics[width=1\textwidth, angle=0]{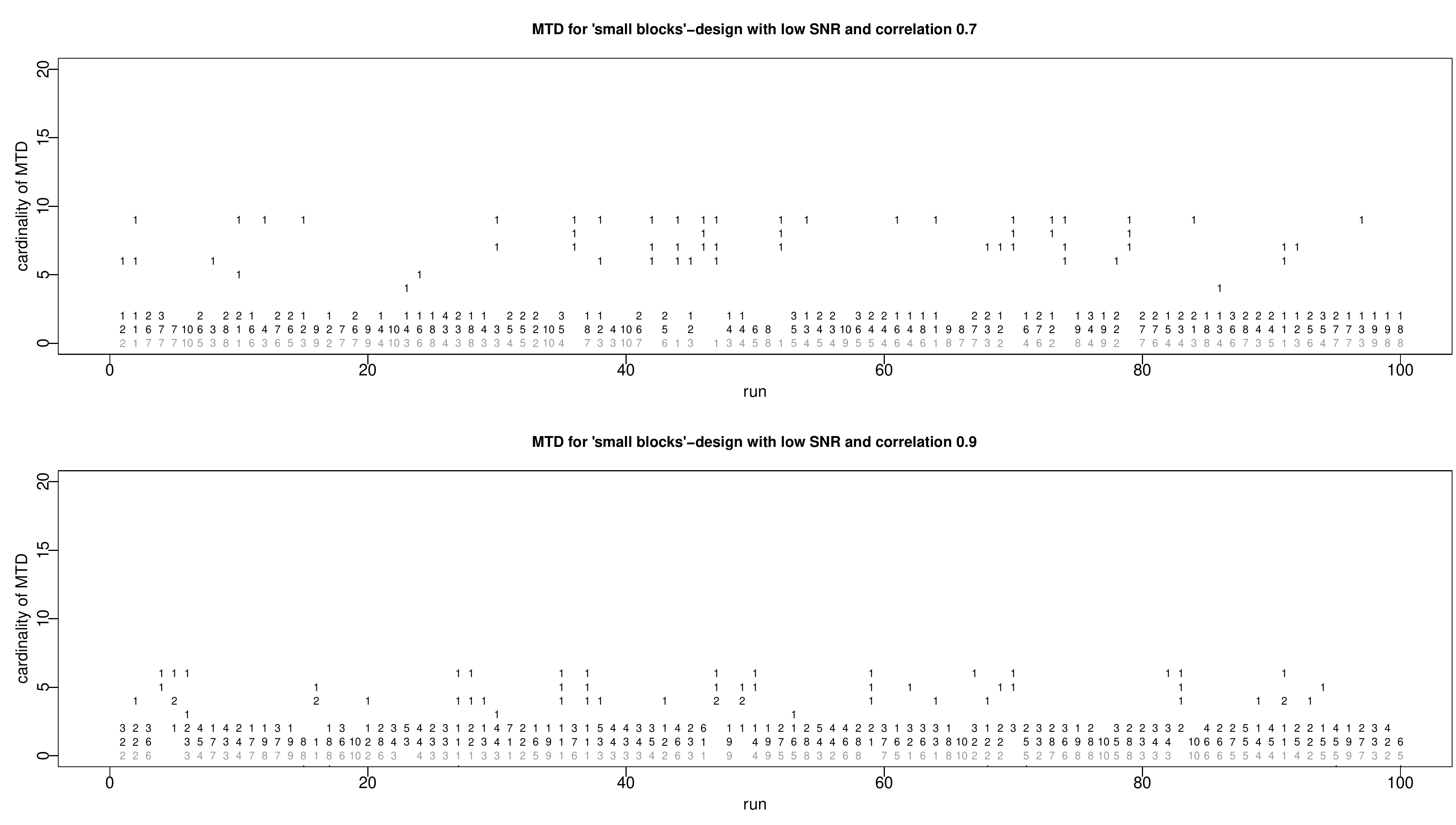}}
\caption{Number of MTDs for ``small blocks''-design with high SNR (SNR=4)
  and $\rho=0.7$ resp.
  $\rho=0.9$. For each of the 100 simulation runs (x-axis) and every cardinality
  (y-axis), the number of MTDs for the hierarchical method with
  \texttt{hclus} clustering (in black) and for the single variable method
  (in gray, for graphical convenience at the bottom of the y-axis).}
\label{figuresmallblockslowSNR2}
\end{sidewaysfigure}

\begin{sidewaysfigure}[!htp]
\centerline{\includegraphics[width=1\textwidth, angle=0]{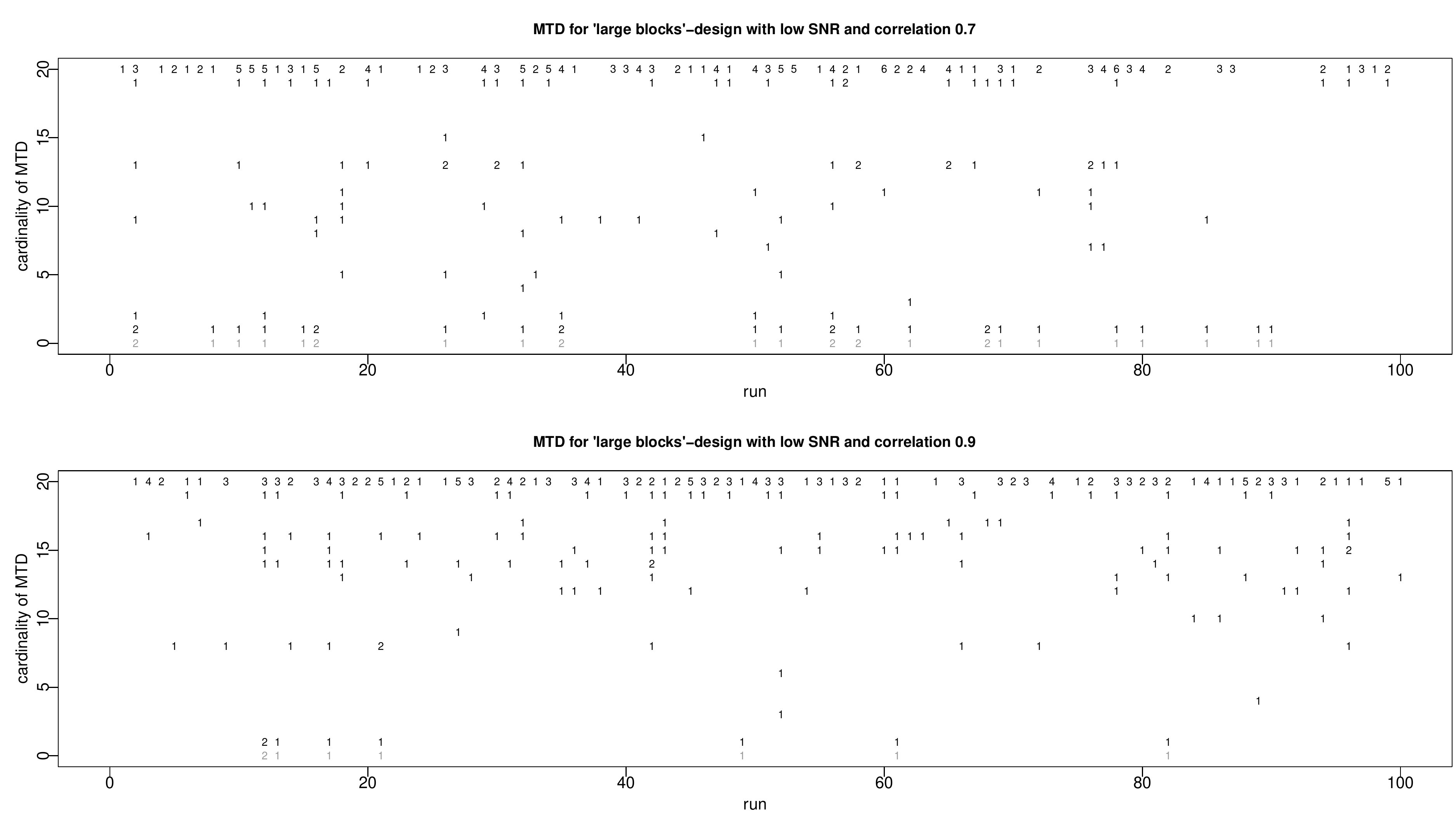}}
\caption{Number of MTDs for ``large blocks''-design with high SNR (SNR=4) and
  $\rho=0.7$ resp.
  $\rho=0.9$. For each of the 100 simulation runs (x-axis) and every cardinality
  (y-axis), the number of MTDs for the hierarchical method with
  \texttt{hclus} clustering (in black) and for the single variable method
  (in gray, for graphical convenience at the bottom of the y-axis).}
\label{figurebigblockslowSNR2}
\end{sidewaysfigure}

Finally, we illustrate in Figure \ref{figureTPRFPRlowSNR} the true
  positive (TPR) rates and false positive rates (FPR) of the Lasso, the
  single variable and the hierarchical method with \texttt{hclus}
  clustering as points in the ROC space.

\begin{figure*}[!htp]
\centerline{\includegraphics[width=0.95\textwidth, angle=0]{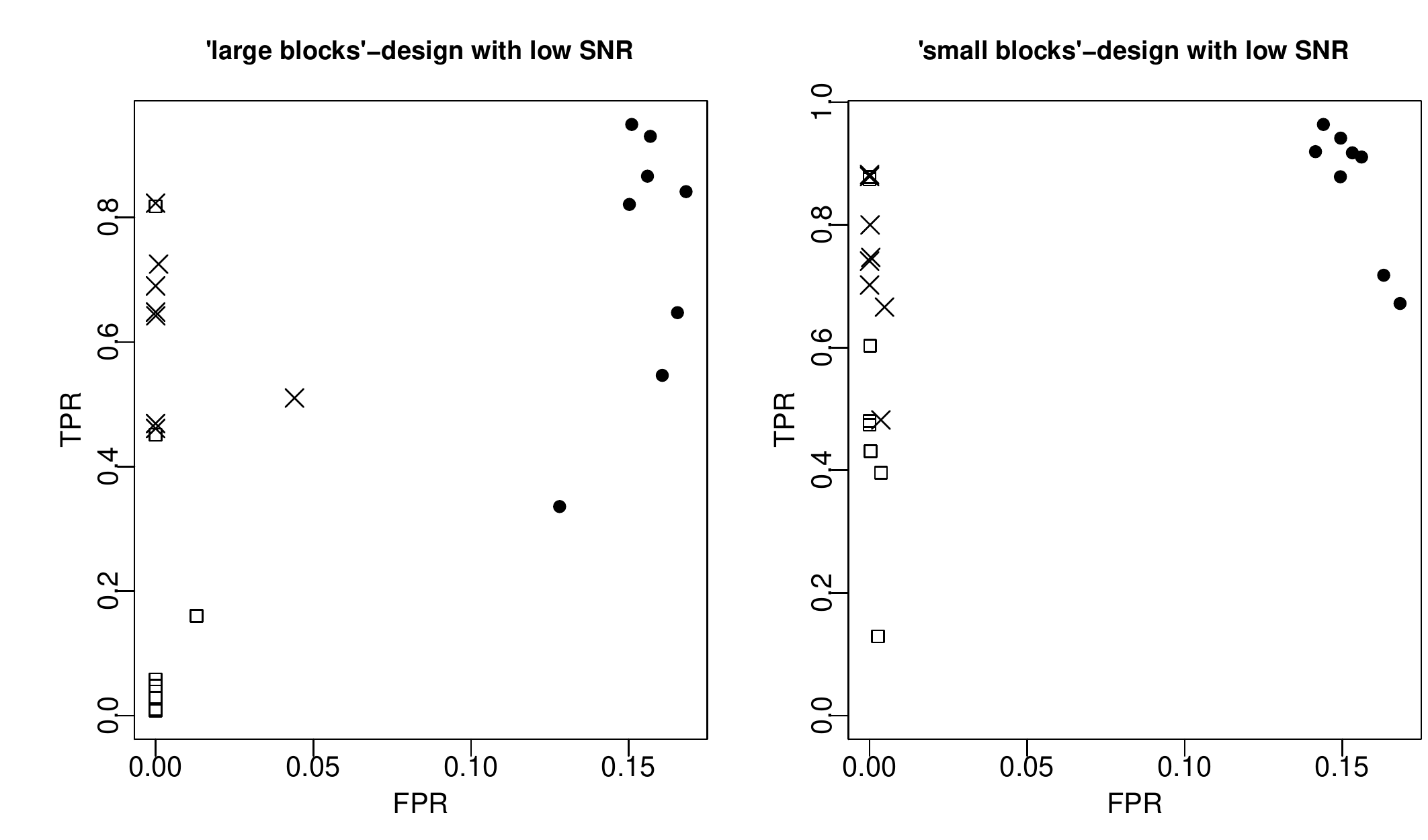}}
\vspace{-0.5cm}
\caption{True positive rate (TPR) and false positive rate (FPR) for the Lasso
  (bullet), the single variable method (box), and the hierarchical method with
  \texttt{hclust} clustering (cross) for different scenarios as indicated
  in the header of the plots.}
\label{figureTPRFPRlowSNR}
\end{figure*}

Comparing Figure \ref{figureTPRFPRlowSNR} with Figure
  \ref{figureTPRFPRhighSNR} in the main paper, we see that the negative
  impact of low SNR is more striking on the TPR then on the FPR which
  remains very similar. Regarding a comparison of the methods, the same
  conclusions as for high $\mbox{SNR}=8$ can be drawn: the single variable
  and hierarchical 
  method do much better than the Lasso in terms of FPR. The price one
  has to pay for the higher reliability is a lower TPR and the hierarchical
method improves the TPR of the single variable method to the level of the
Lasso (when considering MTDs). 

\section*{Proofs}

\subsection*{Proof of Theorem \ref{theo1}}
Our proof is following ideas from the proofs of Theorem 3.1-3.2 in
\citet{memepb09} and the proof Theorem 1 in \citet{Meins08}. 

\bigskip\noindent
\textbf{Proof of first assertion of Theorem \ref{theo1}.}\\
First note that   
\begin{eqnarray*}
\PP(\mathcal{T}^\gamma_{rej} \cap \mathcal{T}_0 \neq \emptyset) &=&
\PP(\exists C \in \mathcal{T}_0 \,:\, Q^C_h(\gamma) \leq \alpha)\\
&=& \PP(\exists C \in \tilde{\mathcal{T}}_0 \,:\, Q^C_h(\gamma) \leq \alpha)
\end{eqnarray*}
where $\tilde{\mathcal{T}}_0$ is the set of all clusters which fulfill the
null hypothesis and are maximal in the sense that
$$\tilde{\mathcal{T}}_0 := \{ C \in \mathcal{T}_0 \,:\, \nexists D \in
\mathcal{T}_0  \mbox{ with } C \subset D \}.$$
This holds, since a direct consequence of the definition of the
hierarchically adjusted p-values $Q^C_h(\cdot)$ is that $Q^{C'}_h(\gamma) \leq
Q^{C}_h(\gamma) \mbox{ for } C \subset C'$ and hence an error committed on a
cluster $C \in \mathcal{T}_0 \backslash \tilde{\mathcal{T}}_0$ implies an
error in a set $C' \in \tilde{\mathcal{T}}_0$, with $C \subset C'$. Moreover,
since $Q^C_h(\gamma) \geq Q^C(\gamma)$,
$$\PP(\exists C \in \tilde{\mathcal{T}}_0 \,:\, Q^C_h(\gamma) \leq \alpha)
\leq \PP(\exists C \in \tilde{\mathcal{T}}_0 \,:\, Q^C(\gamma) \leq
\alpha) = \PP(\min_{C \in \tilde{\mathcal{T}}_0} Q^C(\gamma) \leq \alpha).$$
Hence it remains to show that 
$$\PP(\min_{C \in \tilde{\mathcal{T}}_0} Q^C(\gamma) \leq \alpha) \leq
\alpha.$$
We consider the event $$\mathcal{A}=\{\,\hat{S}^{(b)} \supseteq S_0, \forall\,b=1
\dots B \,\}$$ where all screenings are satisfied. Because of the
$\delta$-screening assumption it holds $$\\P(\mathcal{A}) \geq (1-\delta)^B.$$
In the following we omit the function $\min \{1,\cdot\}$ from the definition of
$Q^C(\gamma)$ in order to simplify the notation (this is possible since the
level $\alpha$ is smaller than 1). Define for $u \in (0,1)$ the function 
$$\pi^C(u):= \frac{1}{B} \sum_{b=1}^B 1 \{ p^{C,(b)}_{adj}
\leq u \}.$$ Then it holds 
\begin{eqnarray*}
Q^C(\gamma) \leq \alpha &\Longleftrightarrow& q_\gamma(\{ p^{C,(b)}_{adj} / \gamma;\,b=1,\dots,B\}) \leq \alpha\\
&\Longleftrightarrow& q_\gamma(\{ p^{C,(b)}_{adj};\,b=1,\dots,B\}) \leq \alpha\gamma\\
&\Longleftrightarrow& \sum_{b=1}^B 1 \{ p^{C,(b)}_{adj} \leq \alpha\gamma \} \geq B\gamma\\
&\Longleftrightarrow& \pi^C(\alpha\gamma) \geq \gamma.
\end{eqnarray*}
Thus
\begin{eqnarray*}
\PP(\min_{C \in \tilde{\mathcal{T}}_0} Q^C(\gamma) \leq \alpha) &\leq&
\sum_{C \in \tilde{\mathcal{T}}_0} \EE ( 1\{ Q^C(\gamma) \leq \alpha \} )\\
&=& \sum_{C \in \tilde{\mathcal{T}}_0} \EE ( 1\{ \pi^C(\alpha\gamma) \geq
\gamma \} )\\
&\leq& \frac{1}{\gamma} \sum_{C \in \tilde{\mathcal{T}}_0} \EE (
\pi^C(\alpha\gamma) ),
\end{eqnarray*}
where for the last inequality a Markov inequality was used. Now, using the
definition of $\pi^C(\cdot)$,
\begin{eqnarray*}
\frac{1}{\gamma} \sum_{C \in \tilde{\mathcal{T}}_0} \EE (
\pi^C(\alpha\gamma) ) &=& \frac{1}{\gamma} \sum_{C \in
  \tilde{\mathcal{T}}_0} \EE ( \frac{1}{B} \sum_{b=1}^B 1 \{ p^{C,(b)}_{adj} \leq \alpha\gamma \} )\\
&=& \frac{1}{\gamma}\frac{1}{B} \sum_{b=1}^B \sum_{C \in
  \tilde{\mathcal{T}}_0} \EE ( 1 \{ p^{C,(b)}_{adj} \leq \alpha\gamma \} )\\
&=& \frac{1}{\gamma}\frac{1}{B} \sum_{b=1}^B \sum_{C \in
  \tilde{\mathcal{T}}_0,\, \hat{S}^{(b)} \cap C \neq \emptyset} 
\EE ( 1 \{ p^{C,(b)}_{adj} \leq \alpha\gamma \} )
\end{eqnarray*}
where the last equality holds since $p^{C,(b)}_{adj} = 1$
if $\hat{S}^{(b)} \cap C = \emptyset$.
Now, for $C$ such that $\hat{S}^{(b)} \cap C \neq \emptyset$ and on $\mathcal{A}$
\begin{eqnarray*}
\EE ( 1 \{ p^{C,(b)}_{adj} \leq \alpha\gamma \} ) &=&
\PP ( p^{C,(b)}_{adj} \leq \alpha\gamma )\\
&=& \PP \Big(p^{C,(b)} \frac{|\hat{S}^{(b)}|}{|C \cap
  \hat{S}^{(b)}|} \leq \alpha\gamma \Big)\\
&=& \PP \Big(p^{C,(b)} \leq \alpha\gamma\frac{|C \cap
  \hat{S}^{(b)}|}{|\hat{S}^{(b)}|} \Big)\\
&\leq& \alpha\gamma\frac{|C \cap \hat{S}^{(b)}|}{|\hat{S}^{(b)}|}.
\end{eqnarray*}
This is a consequence of the uniform distribution of the p-values $p^{C
  \cap \hat{S}^{(b)}}$ given $S \subseteq \hat{S}^{(b)}$ and the sample
split $\{1,\dots,n \} = N_{in}^{(b)} \sqcup N_{out}^{(b)}$. We can hence
conclude that on $\mathcal{A}$
\begin{eqnarray*}
\PP(\min_{C \in \tilde{\mathcal{T}}_0} Q^C(\gamma) \leq \alpha) &\leq&
\frac{1}{\gamma}\frac{1}{B} \sum_{b=1}^B \sum_{C \in
  \tilde{\mathcal{T}}_0,\, \hat{S}^{(b)} \cap C \neq \emptyset}
\alpha\gamma\frac{|C \cap \hat{S}^{(b)}|}{|\hat{S}^{(b)}|}\\
&=& \alpha \frac{1}{B} \sum_{b=1}^B \frac{1}{|\hat{S}^{(b)}|}\sum_{C \in
  \tilde{\mathcal{T}}_0,\, \hat{S}^{(b)} \cap C \neq \emptyset}
|C \cap \hat{S}^{(b)}|\\
&\leq& \alpha \frac{1}{B} \sum_{b=1}^B 1 \leq \alpha,
\end{eqnarray*}
since by definition the sets in $\tilde{\mathcal{T}}_0$ are disjoint and
hence $$\sum_{C \in
  \tilde{\mathcal{T}}_0,\, \hat{S}^{(b)} \cap C \neq \emptyset}
|C \cap \hat{S}^{(b)}| \leq |\hat{S}^{(b)}|.$$
Finally we have 
\begin{eqnarray*}
&&\PP(\min_{C \in \tilde{\mathcal{T}}_0} Q^C(\gamma) \leq \alpha) =\\
&=&\PP(\min_{C \in \tilde{\mathcal{T}}_0} Q^C(\gamma) \leq \alpha \,|\,
\mathcal{A})\,P(\mathcal{A}) + \PP(\min_{C \in \tilde{\mathcal{T}}_0}
Q^C(\gamma) \leq \alpha \,|\,\mathcal{A}^c)\,P(\mathcal{A}^c)\\
&\leq& \alpha + 1-(1-\delta)^B
\end{eqnarray*}

\bigskip\noindent
\textbf{Proof of second assertion of Theorem \ref{theo1}.}\\
We show that 
$$\PP(\exists C \in \mathcal{T}_0\,:\,P^C_h \leq \alpha) \leq \alpha.$$
Defining $\tilde{\mathcal{T}}_0$ as in the proof of Theorem \ref{theo1}
and using similar arguments as there we obtain
$$\PP(\exists C \in \mathcal{T}_0\,:\,P^C_h \leq \alpha) = \PP(\exists C \in
\tilde{\mathcal{T}}_0\,:\,P^C_h \leq \alpha) \leq \PP(\exists C \in
\tilde{\mathcal{T}}_0\,:\,P^C \leq \alpha).$$
As in the proof of Theorem \ref{theo1} we consider the
event $$\mathcal{A}=\{\,\hat{S}^b \supseteq S_0, \forall\,b=1 \dots B
\,\}$$ with $P(\mathcal{A}) \geq (1-\delta)^B$.
The uniform distribution of the p-values 
$p_{\mathrm{partial \: F-test}}^{C \cap \hat{S}^{(b)}}$ given 
$S \subseteq \hat{S}^{(b)}$ and the sample split $\{1,\dots,n \} =
N_{in}^{(b)} \sqcup N_{out}^{(b)}$, together with the fact that sets in
$\hat{S}^{(b)}$ are disjoint, provides on $\mathcal{A}$
$$\EE \Big( \frac{1 \{ p^{C,(b)} \leq \alpha \gamma
  \}}{\gamma} \Big) = \frac{1}{\gamma} \PP ( p^{C,(b)}
\leq \alpha \gamma ) \leq \alpha.$$
Moreover, on $\mathcal{A}$
\begin{eqnarray*}
\EE \Big( \max_{C \in \tilde{\mathcal{T}}_0} \frac{1 \{ p^{C,(b)}_{adj} \leq \alpha \gamma \}}{\gamma} \Big) &\leq& \EE \Big(
 \sum_{C \in \tilde{\mathcal{T}}_0} \frac{1 \{ p^{C,(b)}_{adj}
  \leq \alpha \gamma \}}{\gamma} \Big)\\
&\leq& \EE \Big( \sum_{C \in \tilde{\mathcal{T}}_0,\, \hat{S}^{(b)} \cap C
  \neq \emptyset} \frac{1 \{ p^{C,(b)}_{adj} \leq \alpha \gamma
  \}}{\gamma} \Big)\\
&=& \frac{1}{\gamma} \sum_{C \in \tilde{\mathcal{T}}_0,\,
  \hat{S}^{(b)} \cap C \neq \emptyset} \PP (1 \{ p^{C,(b)}_{adj}
\leq \alpha \gamma \})\\
&\leq& \frac{1}{\gamma} \sum_{C \in \tilde{\mathcal{T}}_0,\,
  \hat{S}^{(b)} \cap C \neq \emptyset} \PP \Big(p^{C,(b)} \frac{|\hat{S}^{(b)}|}{|C \cap
  \hat{S}^{(b)}|} \leq \alpha\gamma \Big)\\
&\leq& \frac{1}{\gamma} \sum_{C \in \tilde{\mathcal{T}}_0,\,
  \hat{S}^{(b)} \cap C \neq \emptyset} \frac{|C \cap
  \hat{S}^{(b)}|}{|\hat{S}^{(b)}|} \alpha\gamma \leq \alpha
\end{eqnarray*}
For a random variable $U$ taking values in $[0,1]$,
$$\sup_{\gamma \in (\gamma_{\min},1)} \frac{1 \{U \leq
  \alpha\gamma \}}{\gamma} = \left\{ \begin{array}{ll}
0, & U \geq \alpha\\
\alpha / U, & \alpha\gamma_{\min} \leq U < \alpha\\
1 / \gamma_{\min}, & U \leq \alpha\gamma_{\min}.\\
\end{array} \right.$$
and if $U$ has an uniform distribution on $[0,1]$
\begin{eqnarray*}
\EE \Big( \sup_{\gamma \in (\gamma_{\min},1)} \frac{1 \{U \leq
  \alpha\gamma \}}{\gamma} \Big) &=& \int_0^{\alpha\gamma_{\min}}
\gamma_{\min}^{-1} dx + \int_{\alpha\gamma_{\min}}^\alpha \alpha x^{-1} dx\\
&=& \gamma_{\min}^{-1} x \big|_{x=0}^{x=\alpha\gamma_{\min}} + \alpha \log
x \big|_{x=\alpha\gamma_{\min}}^{x=\alpha}\\
&=& \alpha + \alpha (\log\alpha - \log(\alpha\gamma_{\min}))\\
&=& \alpha \big(1-\log \frac{\alpha}{\alpha\gamma_{\min}} \big)\\
&=& \alpha (1 - \log \gamma_{\min}).
\end{eqnarray*}
We apply this using as $U$ the uniform distributed 
$p_{\mathrm{partial \: F-test}}^{C \cap \hat{S}^{(b)}}$ and obtain that on $\mathcal{A}$
$$ \EE \Big( \sup_{\gamma \in (\gamma_{\min},1)} \frac{1 \{p^{C,(b)} \leq
  \alpha\gamma \}}{\gamma} \Big) \leq \alpha (1 - \log \gamma_{\min}),$$
and similarly as above
$$ \sum_{C \in \tilde{\mathcal{T}}_0}\EE \Big( \sup_{\gamma \in (\gamma_{\min},1)} \frac{1 \{p^{C,(b)}_{adj} \leq
  \alpha\gamma \}}{\gamma} \Big) \leq \alpha (1 - \log \gamma_{\min}).$$
We can now consider the average over all random splits
$$ \sum_{C \in \tilde{\mathcal{T}}_0}\EE \Big( \sup_{\gamma \in
  (\gamma_{\min},1)} \frac{(1/B) \sum_{b=1}^B 1 \{p^{C,(b)}_{adj} \leq
  \alpha\gamma \}}{\gamma} \Big) \leq \alpha (1 - \log \gamma_{\min})$$
and defining $\pi^C(\cdot)$ as in the proof of Theorem \ref{theo1} and
using a Markov inequality
$$\sum_{C \in \tilde{\mathcal{T}}_0} \EE( \sup_{\gamma \in
  (\gamma_{\min},1)} 1 \{ \pi^C(\alpha\gamma) \geq \gamma \}) \leq \alpha
(1 - \log \gamma_{\min}).$$
We use now the fact, that the events $\{ Q^C(\gamma) \leq \alpha \}$ and
$\{ \pi^C(\alpha\gamma) \geq \gamma \}$ are equivalent and deduce that on $\mathcal{A}$
\begin{eqnarray*}
\sum_{C \in \tilde{\mathcal{T}}_0} \PP (\inf_{\gamma \in
  (\gamma_{\min},1)} Q^C(\gamma) \leq \alpha) &\leq& \alpha
(1 - \log \gamma_{\min}),
\end{eqnarray*}
therefore on $\mathcal{A}$
\begin{eqnarray*}
\PP(\exists C \in
\tilde{\mathcal{T}}_0\,:\,P^C \leq \alpha) &=& \PP
(\min_{C \in \tilde{\mathcal{T}}_0} P^C \leq \alpha)\\
&\leq& \sum_{C \in \tilde{\mathcal{T}}_0} \PP( P^C \leq \alpha )\\
&\leq&\sum_{C \in \tilde{\mathcal{T}}_0} \PP (\inf_{\gamma \in
  (\gamma_{\min},1)} Q^C(\gamma)(1 - \log \gamma_{\min}) \leq \alpha)\\ 
&\leq& \alpha.
\end{eqnarray*}
Finally
\begin{eqnarray*}
&&\PP(\exists C \in
\tilde{\mathcal{T}}_0\,:\,P^C \leq \alpha) =\\
&=&\PP(\exists C \in
\tilde{\mathcal{T}}_0\,:\,P^C \leq \alpha \,|\,
\mathcal{A})\,P(\mathcal{A}) + \PP(\exists C \in
\tilde{\mathcal{T}}_0\,:\,P^C \leq \alpha \,|\,\mathcal{A}^c)\,P(\mathcal{A}^c)\\
&\leq& \alpha + 1-(1-\delta)^B
\end{eqnarray*}
and the proof is concluded.

\subsection*{Proof of Theorem \ref{theoShaffer}}
As the only change to be considered with respect to Theorem \ref{theo1} is
the Shaffer multiplicity adjustment 
(\ref{Shaffermultiplicity}), it suffices to show that 
\begin{equation*}
\sum_{C \in \tilde{\mathcal{T}}_0,\, \hat{S}^{(b)} \cap C \neq \emptyset}
|C|^{\hat{S}^{(b)}}_{e\!f\!f} \leq |\hat{S}^{(b)}|.
\end{equation*}
It holds
\begin{eqnarray*}
&&\sum_{C \in \tilde{\mathcal{T}}_0,\, \hat{S}^{(b)} \cap C \neq \emptyset}
|C|^{\hat{S}^{(b)}}_{e\!f\!f} = \sum_{C \in \tilde{\mathcal{T}}_0,\,
  \hat{S}^{(b)} \cap C \neq \emptyset} \Big( |C \cap \hat{S}^{(b)}| +\\ 
&&+|\si(C)
\cap \hat{S}^{(b)}|\,1\{ \nexists E \in \ch(\si(C)) \mbox{ s.t. } 
E \cap \hat{S}^{(b)} \neq \emptyset \} \Big). 
\end{eqnarray*}
As noted in the proof of Theorem \ref{theo1} the sets in
$\tilde{\mathcal{T}}_0$ are disjoint. Moreover for any cluster $D \in
\tilde{\mathcal{T}}_0$ with $\si(D) \neq \emptyset$, $H_{0,\si(D)}$ is false,
otherwise because of the assumption that $\mathcal{T}$ is binary
$H_{0,\pa(D)}$ would also be true, leading to a contradiction to $D \in
\tilde{\mathcal{T}}_0$. Consider now two sets $C, D \in
\tilde{\mathcal{T}}_0$ with $\hat{S}^{(b)} \cap C \neq \emptyset$ and
$\hat{S}^{(b)} \cap D \neq \emptyset$, since $H_{0,C}$ is true and
$H_{0,\si(D)}$ is false it must be $\si(D) \not\subset C$. On the other hand 
if $C \subset \si(D)$ then the term $|\si(D) \cap \hat{S}^{(b)}|$ wouldn't
be considered in the sum, hence only disjoint $C$ and $\si(D)$ are
considered in the sum. Finally, suppose that for two sets $C, D \in
\tilde{\mathcal{T}}_0$ with $\hat{S}^{(b)} \cap C \neq \emptyset$ and
$\hat{S}^{(b)} \cap D \neq \emptyset$ it is
$\si(C) \subset \si(D)$. Then if $|\si(D) \cap \hat{S}^{(b)}|$ is considered
in the sum it must be $\si(C) \cap \hat{S}^{(b)} = \emptyset$. Putting all
this together we conclude that all sets giving nontrivial contributions to
the sum are disjoint.

\subsection*{Proof of Theorem  \ref{theo:robust}}
In order to prove Theorem \ref{theo:robust} we introduce four Lemmas.
\begin{lemm}\label{robulem1}
$$\big( \hat{\beta}_{I_2}^{\hat{S}} - \beta_{\hat{S}}^0 \big) \sim \mathcal{N}\Big(\big(\bx_{I_2}^{\hat{S},T}
\bx_{I_2}^{\hat{S}}\big)^{-1} \bx_{I_2}^{\hat{S},T}\bx_{I_2}^{\hat{S}^c}\beta_{\hat{S}^c}^0,\sigma^2\big(\bx_{I_2}^{\hat{S},T}
\bx_{I_2}^{\hat{S}}\big)^{-1}\Big)$$
\end{lemm}

\begin{proof}
By definition
\begin{eqnarray*}
\hat{\beta}_{I_2}^{\hat{S}} &=& \big(\bx_{I_2}^{\hat{S},T}
\bx_{I_2}^{\hat{S}}\big)^{-1} \bx_{I_2}^{\hat{S},T} Y_{I_2} = \big(\bx_{I_2}^{\hat{S},T}
\bx_{I_2}^{\hat{S}}\big)^{-1} \bx_{I_2}^{\hat{S},T} \big(
\bx_{I_2}^{\hat{S}} \beta_{\hat{S}}^0 + \bx_{I_2}^{\hat{S}^c}
\beta_{\hat{S}^c}^0  + \eps_{I_2}  \big)\\
&=& \beta_{\hat{S}}^0 + \big(\bx_{I_2}^{\hat{S},T}
\bx_{I_2}^{\hat{S}}\big)^{-1} \bx_{I_2}^{\hat{S},T}\bx_{I_2}^{\hat{S}^c}
\beta_{\hat{S}^c}^0 + \big(\bx_{I_2}^{\hat{S},T}
\bx_{I_2}^{\hat{S}}\big)^{-1} \bx_{I_2}^{\hat{S},T}\eps_{I_2} 
\end{eqnarray*}
and $\hat{\beta}_{I_2}^{\hat{S}}$ is as linear transformation of a normal
distributed random variable also normal distributed. From the formula above
its is easy to see that the expected value
$\big( \hat{\beta}_{I_2}^{\hat{S}} - \beta_{\hat{S}}^0 \big)$ is
$\big(\bx_{I_2}^{\hat{S},T} \bx_{I_2}^{\hat{S}}\big)^{-1}
\bx_{I_2}^{\hat{S},T}\bx_{I_2}^{\hat{S}^c} \beta_{\hat{S}^c}^0$. For the
covariance we can calculate
\begin{eqnarray*}
\Cov \big( \hat{\beta}_{I_2}^{\hat{S}} - \beta_{\hat{S}}^0 \big) &=& \Cov
\big( \big(\bx_{I_2}^{\hat{S},T}
\bx_{I_2}^{\hat{S}}\big)^{-1} \bx_{I_2}^{\hat{S},T}\bx_{I_2}^{\hat{S}^c}
\beta_{\hat{S}^c}^0 + \big(\bx_{I_2}^{\hat{S},T}
\bx_{I_2}^{\hat{S}}\big)^{-1} \bx_{I_2}^{\hat{S},T}\eps_{I_2} \big)\\
&=& \big( \bx_{I_2}^{\hat{S},T}
\bx_{I_2}^{\hat{S}}\big)^{-1} \bx_{I_2}^{\hat{S},T} \Cov (\eps_{I_2})
\bx_{I_2}^{\hat{S}} \big(\bx_{I_2}^{\hat{S},T}
\bx_{I_2}^{\hat{S}}\big)^{-1}\\ 
&=& \sigma^2 \big( \bx_{I_2}^{\hat{S},T}
\bx_{I_2}^{\hat{S}}\big)^{-1}
\end{eqnarray*}
\end{proof}

\begin{lemm}\label{robulem2}
$P_{I_2}^{\hat{S}}$ resp. $Q_{I_2}^{\hat{S}}$ is an orthogonal projection
of $\RR^{|I_2|}$ in $\RR^{|\hat{S}|}$ resp. $\RR^{|I_2|-|\hat{S}|}$.
\end{lemm}

\begin{proof}
It follows from the definition of $P_{I_2}^{\hat{S}}$ and
$Q_{I_2}^{\hat{S}}$, that they satisfy the equation $X^T=X=X^2$. Moreover
\begin{eqnarray*}
\tr(P_{I_2}^{\hat{S}}) &=& \tr(\bx_{I_2}^{\hat{S}} \big(\bx_{I_2}^{\hat{S},T}
\bx_{I_2}^{\hat{S}}\big)^{-1} \bx_{I_2}^{\hat{S},T}) =
\tr(\big(\bx_{I_2}^{\hat{S},T} \bx_{I_2}^{\hat{S}}\big)^{-1} 
\bx_{I_2}^{\hat{S},T}\bx_{I_2}^{\hat{S}})=\\
&=& \tr(I_{|\hat{S}|})=|\hat{S}|\\
\tr(Q_{I_2}^{\hat{S}}) &=& \tr(I_{|I_2|}-P_{I_2}^{\hat{S}}) =
\tr(I_{|I_2|})-(P_{I_2}^{\hat{S}}) = |I_2|-|\hat{S}|
\end{eqnarray*}
and this concludes the proof.
\end{proof}

\begin{lemm}\label{robulem3}
$(\hat{\sigma}_{I_2}^{\hat{S}})^2$ and $\hat{\beta}_{I_2}^{\hat{S}}$ are independent.
\end{lemm}

\begin{proof}
We show that $\hat{\eps}_{I_2}^{\hat{S}}$ and $\hat{Y}_{I_2}^{\hat{S}}$ are
uncorrelated, then the Lemma follows because of
\begin{eqnarray*}
\hat{\beta}_{I_2}^{\hat{S}} &=& \big(\bx_{I_2}^{\hat{S},T}
\bx_{I_2}^{\hat{S}}\big)^{-1} \bx_{I_2}^{\hat{S},T} Y_{I_2}
 = \big(\bx_{I_2}^{\hat{S},T} \bx_{I_2}^{\hat{S}}\big)^{-1}
 \bx_{I_2}^{\hat{S},T} P_{I_2}^{\hat{S}} Y_{I_2}\\  
&=& \big(\bx_{I_2}^{\hat{S},T} \bx_{I_2}^{\hat{S}}\big)^{-1}
\bx_{I_2}^{\hat{S},T} \hat{Y}_{I_2}^{\hat{S}}
\end{eqnarray*}
and the fact that the random variables involved are normally distributed.
\begin{eqnarray*}
\Cov \big(\hat{\eps}_{I_2}^{\hat{S}},\hat{Y}_{I_2}^{\hat{S}} \big) &=& \Cov
\big(Q_{I_2}^{\hat{S}} Y_{I_2},P_{I_2}^{\hat{S}} Y_{I_2} \big) =
\Cov \big( Y_{I_2} \big) Q_{I_2}^{\hat{S}} P_{I_2}^{\hat{S},T}\\
&=& \sigma^2 \big( I_{I_2}-P_{I_2}^{\hat{S}} \big) P_{I_2}^{\hat{S}} =
\sigma^2 \big( P_{I_2}^{\hat{S}} - (P_{I_2}^{\hat{S}})^2 \big) = 0
\end{eqnarray*}
\end{proof}

\begin{lemm}\label{robulem4}
$(\hat{\sigma}_{I_2}^{\hat{S}})^2$ is an unbiased estimator of $\sigma^2$
and
$$ \big( |I_2| - |\hat{S}| \big) \frac{(\hat{\sigma}_{I_2}^{\hat{S}})^2}{\sigma^2} \sim \chi^2_{|I_2| - |\hat{S}|}$$
\end{lemm}

\begin{proof}
We calculate
\begin{eqnarray*}
\EE \big[ (\hat{\sigma}_{I_2}^{\hat{S}})^2 \big] &=& \frac{1}{|I_2| -
  |\hat{S}|} \, \EE \big[ \hat{\eps}_{I_2}^{\hat{S},T}
\hat{\eps}_{I_2}^{\hat{S}} \big] = \frac{1}{|I_2| -
  |\hat{S}|} \tr \Big( \EE \big[ \hat{\eps}_{I_2}^{\hat{S}}
\hat{\eps}_{I_2}^{\hat{S},T} \big] \Big)\\
&=& \frac{1}{|I_2| - |\hat{S}|} \tr \big( Q_{I_2}^{\hat{S}} \, \EE \big[ Y_{I_2}
Y^T_{I_2} \big] Q_{I_2}^{\hat{S},T}\big)\\ 
&=& \frac{\sigma^2}{|I_2| -
  |\hat{S}|} \tr \big(Q_{I_2}^{\hat{S}}Q_{I_2}^{\hat{S},T} \big) = \sigma^2 
\end{eqnarray*}
To see that the given random variable is chi-square distributed we use a
geometrical approach. Consider a basis of $|I_2|$ orthogonal vectors,
s.t. the first $|\hat{S}|$ vectors span the space given by the vectors of
$\bx_{I_2}^{\hat{S}}$ and call the corresponding transformation matrix
$G$ (the columns of $G$ corresponds the coordinates of the new basis
vectors in the old coordinate system). Then $G$ is orthogonal and using a
star for the new coordinate system we have $Y_{I_2}^*=G^TY_{I_2},~
\eps_{I_2}^*=G^T\eps_{I_2}$. By construction it is
\begin{eqnarray*}
(\hat{Y}^{\hat{S}}_{I_2})^* &=& (Y_1^*, \dots, Y_{|\hat{S}|}^*,0,\dots,0)^T\\
(\hat{\eps}^{\hat{S}}_{I_2})^* &=& (0,\dots,0,\eps_{|\hat{S}|+1}^*,
\dots, \eps_{|I_2|}^*)^T,
\end{eqnarray*}
using the orthogonality of $G$ we get 
\begin{eqnarray*}
\hat{\eps}^{\hat{S},T}_{I_2} \hat{\eps}^{\hat{S}}_{I_2} =
(\hat{\eps}^{\hat{S},T}_{I_2})^* (\hat{\eps}^{\hat{S}}_{I_2})^* =
\sum_{|\hat{S}|+1}^{I_2} \eps_i^{*2}.
\end{eqnarray*}
Again because of the orthogonality of $G$, it holds
$\eps_{I_2}^*=G^T\eps_{I_2} \sim \mathcal{N}(0,\sigma^2I_{I_2})$ and the
proof is concluded.
\end{proof}

\bigskip\noindent
\textbf{Proof of Theorem \ref{theo:robust}.}\\

Theorem \ref{theo:robust} follows from the Lemmas
\ref{robulem1}, \ref{robulem2}, \ref{robulem3} and \ref{robulem4} and the
following considerations. First rewrite
\begin{eqnarray*}
&& \frac{(A\hat{\beta}_{I_2}^{\hat{S}}-A\beta_{\hat{S}}^0)^T
  \big(A\big(\bx_{I_2}^{\hat{S},T}
  \bx_{I_2}^{\hat{S}}\big)^{-1}A^T\big)^{-1}(A\hat{\beta}_{I_2}^{\hat{S}}-A\beta_{\hat{S}}^0
)}{q(\hat{\sigma}_{I_2}^{\hat{S}})^2}\\
&=& \left(\frac{(A\hat{\beta}_{I_2}^{\hat{S}}-A\beta_{\hat{S}}^0)^T
  \big(A\big(\bx_{I_2}^{\hat{S},T}
  \bx_{I_2}^{\hat{S}}\big)^{-1}A^T\big)^{-1}(A\hat{\beta}_{I_2}^{\hat{S}}-A\beta_{\hat{S}}^0
)}{q\sigma^2}\right) \left(
\frac{(\hat{\sigma}_{I_2}^{\hat{S}})^2}{\sigma^2} \right)^{-1}.
\end{eqnarray*}
Because of Lemma \ref{robulem3} the two terms in the big brackets are
independent. Because of Lemma \ref{robulem4} the term in the second big
bracket would be $\chi^2_{|I_2|-|\hat{S}|}$-distributed, if multiplied
by $|I_2|-|\hat{S}|$. Let's consider the term in the first big
bracket. Because of Lemma \ref{robulem1}, the quadratic form given by the
term in the first big bracket multiplied by $q$ corresponds to the
quadratic form $Z^TZ$ where 
\begin{eqnarray*}
Z &=& \frac{1}{\sigma}\big( A\big(\bx_{I_2}^{\hat{S},T}
  \bx_{I_2}^{\hat{S}}\big)^{-1}A^T
  \big)^{-1/2}A(\hat{\beta}_{I_2}^{\hat{S}}-\beta_{\hat{S}}^0) \sim
  \mathcal{N}(\mbox{BIAS},I_q)
\end{eqnarray*}
with
$$\mbox{BIAS} = \frac{1}{\sigma} \big(A \big(\bx_{I_2}^{\hat{S},T}
\bx_{I_2}^{\hat{S}}\big)^{-1} A^T \big)^{-1/2} A \big(\bx_{I_2}^{\hat{S},T}
\bx_{I_2}^{\hat{S}}\big)^{-1}
\bx_{I_2}^{\hat{S},T}\bx_{I_2}^{\hat{S}^c}\beta_{\hat{S}^c}^0$$
and this concludes the proof.

\end{document}